%% file: Monogenic_submitted_8_10_2026.tex
\documentclass[12pt]{amsart}
\pdfoutput=1
\usepackage[protrusion=true,expansion=true]{microtype}
\usepackage[toc,page]{appendix}
\usepackage[all]{xy}
\usepackage{graphicx}
\usepackage{tikz}
\usetikzlibrary{patterns}
\usepackage{enumerate}
\usepackage[utf8]{inputenc}
\usepackage{stmaryrd}
\usepackage[T1]{fontenc}
\usepackage{epigraph}

\reversemarginpar

\newcommand{\Ft}{{\mathcal F}}

\newcommand{\Gt}{{\mathcal G}}

\newcommand{\Mt}{{\mathcal M}}
\newcommand{\Nt}{{\mathcal N}}
\newcommand{\cO}{{\mathcal O}}
\newcommand{\Ot}{{\mathcal O}}
\newcommand{\cP}{{\mathcal P}}
\newcommand{\Pt}{{\mathcal P}}

\newcommand{\Rt}{{\mathcal R}}

\newcommand{\cU}{{\mathcal U}}
\newcommand{\Ut}{{\mathcal U}}

\newcommand{\Vt}{{\mathcal V}}

\newcommand{\C}{{\mathbb C}}
\newcommand{\CM}{{\mathbb C}}

\newcommand{\NM}{{\mathbb N}}
\newcommand{\PM}{{\mathbb P}}
\renewcommand{\P}{{\mathbb P}}

\newcommand{\RM}{{\mathbb R}}

\newcommand{\X}{{\mathbb X}}
\newcommand{\XM}{{\mathbb X}}

\newcommand{\Z}{{\mathbb Z}}

\newcommand{\lra}{\longrightarrow}
\renewcommand{\d}{\partial}
\newcommand{\e}{\varepsilon}
\newcommand{\p}{\varphi}
\renewcommand{\a}{\alpha}
\renewcommand{\b}{\beta}
\newcommand{\g}{\gamma}
\newcommand{\s}{\sigma}

\newcommand{\dt}{\delta}

\renewcommand{\l}{\lambda}
\newtheorem{theorem}{Theorem}

\newtheorem{lemma}{Lemma}

\newtheorem{definition}{Definition}
\newtheorem{proposition}{Proposition}
\newtheorem{corollary}{Corollary}
\usepackage{xcolor}

\allowdisplaybreaks[2]
\usepackage{needspace}
\let\oldsection\section
\renewcommand{\section}{\Needspace{9\baselineskip}\oldsection}
\title[A new class of Borel monogenic functions]{A new class of Borel monogenic functions}
\author{  Mauricio  Garay and Duco van Straten}
\subjclass[2020]{Primary 30D60; Secondary 26E10,37J40}

\begin{document}
\begin{abstract}  Borel's theory of monogenic functions is a generalisation of holomorphic functions defined on locally closed sets, a complex version of Whitney differentiability. In this paper we introduce a class of Borel monogenic functions which are infinite sums of polar parts. This class is relevant to perturbation theory and especially to KAM theory and dynamical systems. We show that under Diophantine conditions on the poles, the sums have a Gevrey
Taylor expansion. In special geometric configurations, this implies that they define a sheaf of quasi-analytic functions, that is, functions uniquely defined by their Taylor series.
\end{abstract}
\maketitle
\setlength{\epigraphwidth}{0.8\textwidth}
\epigraph{{\em Probably, the dependence of $\phi_{mn}(x_1,x_2,\theta)$ on the parameter $\theta$ in $R$ is related to a class of functions of the type of Borel's monogenic functions and, despite its everywhere discontinuous nature, admits an investigation by some appropriate analytic means.}}
{A.N. Kolmogorov, {\em Theory of dynamical systems and classical mechanics}, ICM Amsterdam, 1954.}
\section*{Introduction}
%%%%%%%%%%%%%%%%%%%%%%%%%%%%%%%%%%%%%%%%%%%%

In this paper, we introduce a class of functions defined by sums of holomorphic functions, which is relevant to the study of perturbation theory in dynamical systems or mathematical physics
and for number theory.  In classical analysis there is a large body of results on sums of rational functions, such
as the work of Stieltjes \cite{Stieltjes_1894}, where the Cauchy--Stieltjes integral transforms atomic
measures into meromorphic functions, and the work of Runge \cite{Runge} in the context of approximation theory for which we refer to \cite{Zalcman1968}.

For infinite sums of rational functions, different domains of holomorphy may give rise to different functions. For instance, Weierstrass observed in 1880 that the series
$$\sum_{n\ge0} \frac1{z^n+z^{-n}}
$$
converges for $|z|<1$ and for $|z|>1$, and represents there two different analytic functions, none of which can be continued across the unit circle~\cite{Weierstrass_1880,Weierstrass_1881}. In~\cite{Borel_monogene,Borel1917}, Borel made the crucial observation that if analytic continuation is blocked by the poles, it may exist nevertheless
along rays. This led to the definition of holomorphic functions on locally closed sets, that he called {\em monogenic functions}.

Monogenicity on a locally closed set $X$ generalises the notion of  holomorphicity on open sets. It is defined by
requiring the existence of a function $\a$ such that
$$F(y)=F(x)+\a(x)(y-x)+| y-x|R(x,y),\ R(x,y) \xrightarrow[x,y \to a]{}0 .$$
The function $\a$ is then called the derivative of $F:X \to \C$ and we use the notation $\a=F'$. Higher derivatives are defined as well, so we may talk of $C^k$-monogenic functions. Note that instead of fixing a point $a$ and considering its increment, as one usually does for derivatives,
the direction $y-x$ may vary arbitrarily and the estimate is required to be uniform, which we recognise as the complex version of Whitney differentiability (but it is prior to Whitney's work on the subject), so $C^k$-monogenic functions are the same as $C^k$-Whitney holomorphic functions~\cite{Borel_monogene,Borel1917,Whitney_extension}.\footnote{Many authors consider an increasing sequence of closed sets $K_j$, obtained by letting the radii of the removed discs tend to zero, and define monogenic functions as the projective limit of the spaces of monogenic functions on the $K_j$~\cite{Herman1985,Marmi_Sauzin,Winkler}. We do not follow this convention: once the sheaf of monogenic functions on a closed set is available, taking such a limit is a purely formal operation.}

Borel was able to prove the:
\begin{theorem}[{\cite{Borel_1898}}]
Assume that $\sum_n\sqrt{|A_n|}<\infty$. Then, for every $k>0$, the series $\sum_n\frac{A_n}{z-a_n}$ converges absolutely and uniformly on the complement $R_k$ of the discs $D(a_n,k\sqrt{|A_n|})$. If there are an open set $U$ and a holomorphic function $\varphi$ on $U$ such that the sum of the series is equal to $\varphi$ on $U\cap R_k$ for every $k>0$, then $A_n=0$ for all the points $a_n \in U$.
\end{theorem}
Under the assumptions of the theorem, the decomposition into polar parts is unique: if such a sum is holomorphic near a point, it has no polar part there. We will call this unique decomposition property the {\em Borel property} (Definition~\ref{D::Borel}). For meromorphic functions on an open subset, it reduces to a classical fact of Mittag-Leffler theory: the polar parts of a meromorphic function are uniquely determined, and the difference between the function and the sum of its polar parts is holomorphic.

Of course, the assumptions of Borel's theorem are highly restrictive as they impose that all poles are simple.
For poles of order one, a sum of polar parts is the Cauchy--Stieltjes transform of an atomic measure:
$$\sum_{n}\frac{A_n}{z-a_n}=\int\frac{d\mu(\zeta)}{z-\zeta},\qquad \mu=\sum_n A_n\delta_{a_n} $$
and problems on infinite sums reduce to measure theory. However, in perturbative expansions such as Hamiltonian normal forms~\cite{HNF}, poles are usually of infinite order, and there is no measure to which Borel's theorem could be applied. Borel was very much aware that his theorem was only a first step towards a better understanding:
\begin{quote}
{\em It is not possible for us to give this Chapter a decisive and definitive
conclusion; for, in our opinion, the question examined here has not been
fully resolved and calls for further research. We would be satisfied
if we had convinced our readers that neither the fundamental work of
Weierstrass nor the subsequent work of Mittag-Leffler, Appell, Poincar\'e,
Runge and Painlev\'e fully resolves the question of the relationship
between the notion of an analytic function and that of an analytic
expression. One may even say, without exaggeration, that the study
and classification of analytic expressions which cannot represent
zero without being identically zero remain almost entirely
to be undertaken.}\footnote{Il ne nous est pas possible de donner \`a ce Chapitre une conclusion
d\'ecisive et d\'efinitive; car, \`a notre avis, la question qui y est
examin\'ee n'est pas enti\`erement r\'esolue et appelle de nouvelles
recherches. Nous serions satisfait si nous avions convaincu nos
lecteurs que, ni les travaux fondamentaux de Weierstrass, ni les
travaux ult\'erieurs de MM. Mittag-Leffler, Appell, Poincar\'e, Runge,
Painlev\'e, ne r\'esolvent enti\`erement la question des rapports entre
la notion de fonction analytique et d'expression analytique.
On peut m\^eme dire, sans exag\'eration, que l'\'etude et la classification
des expressions analytiques, non susceptibles de repr\'esenter z\'ero
sans \^etre identiquement nulles, est encore \`a faire \`a peu pr\`es
enti\`erement.}
\hfill\cite{Borel_1898}
\end{quote}

In this paper, we construct a certain sheaf of functions which has the Borel property. Our main observation is to impose not only conditions on the growth of
the terms of the sum but also on the geometry of the set $X$. We decided to call this sheaf {\em meandromorphic functions} because, in the complex situation, they originated from series defining Riemann surfaces which wind around the poles and have the shape of a complex meander.
These conditions are natural in dynamical systems, mathematical physics and number theory; nevertheless, they exclude other monogenic functions such as the P\'erez-Marco Cantor Riemannium~\cite{Perez_Marco_2020_Cantor}.

The property considered by Borel is therefore of a very general nature and is not to be misinterpreted as quasi-analyticity, which is on the contrary an exceptional phenomenon. Inside spaces of monogenic functions, Borel studied {\em quasi-analytic classes}, which means in modern terms a sheaf $\Ft$ on  $X \subset \CM$ and  maps 
$$\Ft_{z_0} \to \CM[[z-z_0]]$$  assigning to a function its Taylor expansion. If the map is injective then the sheaf $\Ft$ is called {\em quasi-analytic.} For instance, if we take $X=\CM$ then the structure sheaf $\Ot_X$ is quasi-analytic as the Taylor expansion of a holomorphic function uniquely
defines the holomorphic function and this holds of course not only in $\CM$ but in any complex manifold. 
 
Borel observed that quasi-analyticity depends on the size of the discs which are removed. If $r_n$ denotes the radius of the $n$-th disc, he proved that his sheaf of monogenic functions is quasi-analytic under the condition~\cite{Borel_monogene,Borel1917}:
$$\lim_{n\to\infty}\frac{n}{\log\log\log(1/r_n)}=0 .$$
Winkler relaxed it to~\cite{Winkler}:
$$\sum_n r_n<\infty,\qquad \sum_n\frac1{\log(1/r_n)}<\infty .$$
For instance, the vector space consisting of series of the form
$$\sum_{n\ge1}\frac{c_n}{1-n z},\ c_n=O(e^{-n^2}) $$
is, according to Winkler, quasi-analytic. Similarly, Marmi and Sauzin showed that the space of series
$$\sum_{\omega}\frac{c_\omega}{z-\omega},\ c_\omega=O(q^{n(\omega)}),\ |q|<1, $$
where $\omega$ runs over the roots of unity lying on a half-circle with endpoint $z_0=e^{2i\pi\theta}$ and $n(\omega)$ denotes the order of $\omega$, is quasi-analytic at $z_0$ if $\theta$ is of constant type~\cite[\S 4.4]{Marmi_Sauzin}.

After the elaboration of function theory on set theory rather than expressions given by explicit formulas,
Borel's work was rapidly forgotten and is largely ignored by modern mathematicians, for whom quasi-analytic classes are nowadays synonymous with Denjoy--Carleman's quasi-analytic classes~\cite{Carleman_1926,Denjoy_1921}.

Instead of conditions on the removed discs, Denjoy proposed to impose direct growth conditions on the derivatives of a real function.
He proved a criterion for quasi-analyticity showing for instance that functions satisfying  $|f^{(n)}(x)| \leq k^nn! \log(n)$
define a quasi-analytic sheaf on the real line. Carleman slightly improved Denjoy's criterion and defined what are now called {\em Denjoy--Carleman classes}.  The success of the Denjoy--Carleman theory can be explained by the simplicity of their results; its defect is that it is too restrictive for concrete applications.
Indeed  Denjoy--Carleman conditions, as well as Borel--Winkler, are highly restrictive.

Unique decomposition explains that quasi-analyticity is an exceptional property: if we decompose the space in various sectors, then on each sector we have a decomposition into polar and holomorphic parts giving rise to Stokes type data; quasi-analyticity imposes that such data are in fact trivial. Indeed function spaces, which are natural for perturbation theory (see for instance~\cite{Versal_fields,HNF}), do not form a quasi-analytic sheaf.

We will nevertheless construct a quasi-analytic space of Borel monogenic functions and relax the Borel--Winkler and Denjoy--Carleman conditions by imposing a geometric structure on $X$. For instance, the space of series
$$\sum_{n\ge1}\frac{c_n}{1-n z},\ c_n=O(q^n),\ |q|<1 $$
 is already quasi-analytic and, as we shall see in the paper, this is due to the fact that the poles lie on a half-line, contrary to the space of series
 $$\sum_{n\ge1}c_n\left(\frac{1}{1-n^2 z}+\frac{1}{1+n^2 z}\right),\ c_n=O(q^n),\ e^{-\pi}<q<1 $$
 which admits flat functions (see \S\ref{SS::flat}). This simple example already shows that the geometry of the set $X$ on which the function
 is defined plays a crucial role.

\subsection*{ Content of the paper}
In the first section, we give a tour of the different phenomena that one encounters when studying infinite sums of rational functions.\\
 In the second section, we set up the basic formalism of {\em perforations} that
we will use for the rest of the paper and define the sheaves of {\em summable} and of {\em meandromorphic} sequences defined by a perforation.\\
 The third section is devoted to first {\em regularity properties} of summable sequences. Using the Cauchy estimates and a shrinking lemma, we prove a Taylor formula.\\
 In section four, we show that meandromorphic sums have a polar decomposition similar to the pole decomposition of rational functions. We show that this decomposition is unique, that is, that meandromorphic sums have the Borel property, when the radius function is an Arnold function, and deduce a criterion for the sum to be holomorphic.\\
In the fifth section, we show that, under appropriate conditions, the sheaf of meandromorphic sums is quasi-analytic, and we observe that our results imply, as a general rule, the divergence of the Taylor expansion of meandromorphic sums. Then we recall basic facts about Gevrey classes and show that meandromorphic functions are in fact of Gevrey class under Diophantine conditions on perforations. We prove the main theorem using Watson's lemma. Combined with the Hamiltonian normal form~\cite{HNF}, this result implies the Gevrey expansions of KAM tori discovered by Popov~\cite{Popov_KAM}.

In the first appendix, we illustrate how the results apply to the Poincar\'e meander. In the second appendix we give some details on the estimates needed for the Euler quantum logarithm. In the third appendix, we show that for sums of rational functions the Arnold condition on the radius function can be dropped.

%%%%%%%%%%%%%%%%%%%%%%%  
\section{Infinite sums of rational functions}
We start with a tour of the different kinds of phenomena that one encounters when studying infinite sums of rational functions.
\subsection{The quantum logarithm}
The first of our examples is due to Euler, who, in 1753, considered the power series~\cite{Euler_log}:\\
\[ s=\frac{1}{a-1}+\frac{z}{aa-1}+\frac{zz}{a^3-1}+\frac{z^3}{a^4-1}+\frac{z^4}{a^5-1}+etc.\]
Up to a multiplication by $z$, the series is now called the {\em quantum logarithm} or simply the {\em $q$-logarithm}~(see e.g. \cite{Sondow_Zudilin}):
$$\log_q(1-z):=\sum_{n>0} \frac{z^n}{q^n-1}. $$
Indeed multiplying this expression by $q-1$ and taking the limit for $q=1$ term by term, we recover the usual logarithmic series:
$$\lim_{q \to 1} (q-1)\log_q(1-z)=\sum_{n\geq 1} \frac{z^n}{n}=-\log(1-z),$$
thus explaining the name. 

Euler's quantum logarithm is the first example in the larger family of $q$-hyper\-geometric functions that have 
poles at the roots of unity, which accumulate on the unit circle in the $q$-plane~\cite{Gasper_Rahman}.
Such and similar series play a role in several areas of mathematics: in combinatorics, number theory,
modular forms, representation theory, in mathematical physics, especially for quantum invariants of knots and
3-manifolds~(see e.g. \cite{Garoufalidis_Zagier,Zagier2001}) and in singularity theory~(see~\cite{Mourtada_Ramanujan}). In this specific case of poles appearing at the
roots of unity, one can try to develop appropriate algebraic techniques to manipulate
such series. Habiro \cite{Habiro2004} introduced an interesting ring, the cyclotomic completion of $\Z[q]$,
where many of these objects find a natural home.
%%%%%%%%%%%%%%%%%%%
\subsection{ Weierstrass and Borel continuation}
To see some of the issues that arise when considering expansions with rational functions as coefficients, we follow Borel~\cite[Chapitre~III, p.~55--68]{Borel1917}, take a closer look at Euler's series and slightly change the notations. We set: 
\[ L(x,z) :=\sum_{n=1}^{\infty} \frac{x^{n-1}}{z^n-1}, \]
and fix a value for $x$ with $|x|<1$, say $x=1/2$. 
The expansion $L(1/2,z)$ is the sum of functions 
$$f_n(z) := \frac{1}{2^{n-1}}\frac{1}{z^n-1},$$
defined on the complement of the set of $n$-th roots of unity 
$\mu_n:=\{ \omega \in \CM: \omega^n=1 \}.$ For $|z|<1$ we have asymptotically:
$$  f_n(z) \sim \frac{-1}{2^{n-1}} ,$$ which shows that the series $\sum_{n=1}^{\infty} f_n(z)$ converges inside the open unit disc and defines a holomorphic function $F_0$ on the open unit disc. For $|z| > 1$ the asymptotic behaviour is 
$$f_n(z) \sim \frac{1}{2^{n-1}z^n}, $$ 
so the series converges
also outside the closed unit disc and defines there a holomorphic function $F_{\infty}$.\\
As the set $\bigcup_n \mu_n$ of all roots of unity is dense on the unit circle there is, in the sense of Weierstrass, no analytic continuation possible from the inside to the outside of the circle and
the unit circle is a natural boundary for the functions $F_0$ and $F_{\infty}$; apparently the functions
$F_0$, $F_{\infty}$ have no relation.   

But another kind of continuation was considered by Borel
leading to his theory of {\em monogenic functions}: continuation along rays. For instance, if we restrict $z$
to the line segment
$$L=\{ t e^{2\pi i \sqrt{2}} \in \CM: t \in [0,2) \} ,$$
then the series converges for each value of $t$ and forms a bridge between the functions $F_0$ and $F_\infty$ (see Appendix~B for details).

 In the following pictures, we depict the
real and imaginary parts of the functions $F_0$ and $F_{\infty}$ defined by the above series restricted to the line $L$.
 We observe that passing through the unit circle ($t=1$) does not reveal any discontinuity.
 
\begin{center}
  \includegraphics[width=0.31\linewidth]{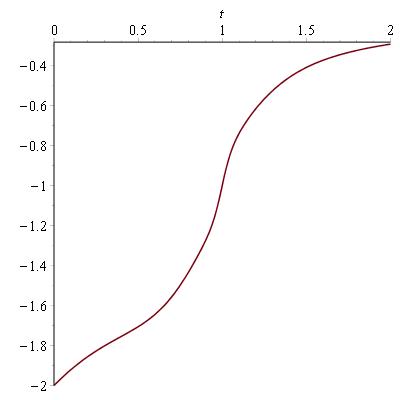} \hspace{.3cm}
  \includegraphics[width=0.31\linewidth]{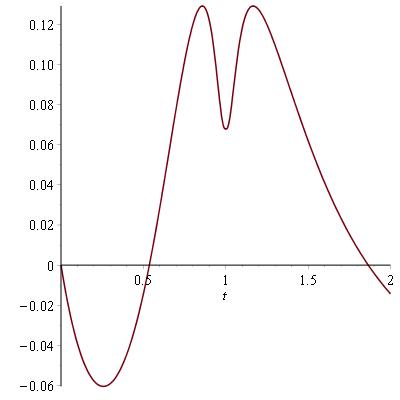}\\
{\em \tiny  Real and imaginary parts of $L(1/2,z)$ for $z=t e^{2\pi i \sqrt{2}}$, $t \in [0,2)$.}
\end{center}
In 1961, Arnold  studied the adjoint action of the diffeomorphism group of the circle. The linearisation of the conjugation equation led Arnold to a $q$-difference equation for which the $q$-logarithm is a fundamental solution~\cite{Arnold_SD1}. Arnold proved the existence of a subset of full measure of the circle on which the $q$-logarithm is monogenic, in particular it is $C^1$ along `most' rays that avoid the poles at the roots of unity~\cite[Lemma 10]{Arnold_SD1}. Herman later proved that the solution of the non-linear conjugacy problem is $C^\infty$-holomorphic in the sense of Whitney with respect to the rotation number~\cite[III.13--14]{Herman1985}.

However, if we replace $\sqrt{2}$ by an irrationality $\alpha$ that is very well approximated by rational numbers, one
can expect divergences and the line segment will be blocked by the unit circle. In this way aspects of Diophantine
approximation enter the discourse in an essential way.   A detailed analysis of the resurgence properties, in the sense of \'Ecalle~\cite{Ecalle_fonctions}, of the quantum logarithm can be found in~\cite{Marmi_Sauzin}. 

%%%%%%%%%%%%%%%%%%%%%%%%%
\subsection{ Poincar\'e's example}
Power series with rational functions as coefficients in relation with dynamical systems already appeared in Poincar\'e's
{\em Les M\'ethodes Nouvelles de la M\'ecanique C\'eleste}~\cite[VIII\ \S 119]{Poincare_Methodes_2}. Poincar\'e considered the series expansion:
$$f(x,z)=\sum_{k \geq 1} \frac{x^k}{1+kz}. $$
Here the poles have only one accumulation point at the origin, so the situation is simpler than that of the quantum logarithm.
Suppose we want to solve $z$ from the implicit equation $z=f(x,z)$. By truncation of the series at $x^n$ we obtain an algebraic
curve defined by the equation
$$z=\sum_{k=1}^n \frac{x^k}{1+k z} .$$
If we plot the real part of this curve we observe that with increasing $n$ the curve winds more and more due to the
presence of more and more of the poles at $z=-1/k$. For this reason we call the right-hand side the {\em Poincar\'e meandromorphic series}.\\
  
\begin{figure}[htb!]
\includegraphics[width=0.45\linewidth]{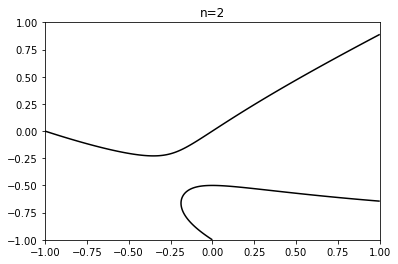}
\includegraphics[width=0.45\linewidth]{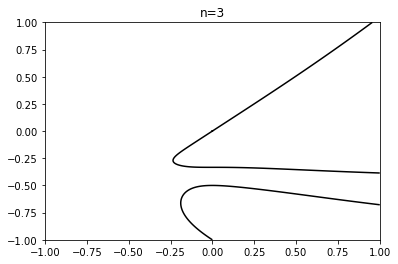}
\end{figure}
\begin{figure}[htb!]
\includegraphics[width=0.45\linewidth]{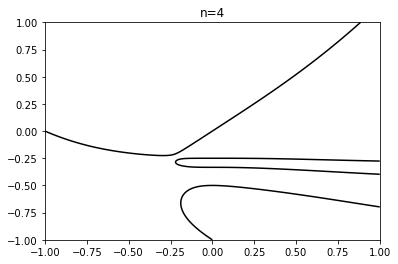}
\includegraphics[width=0.45\linewidth]{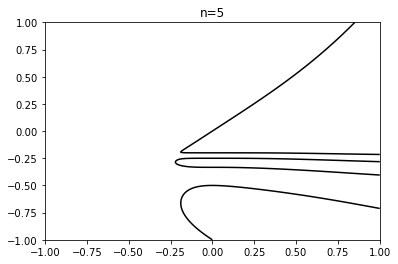}\\
{\small The curves approximating $z=f(x,z)$ in the Poincar\'e example.}
\end{figure}
 
A direct computation shows that, for a fixed value of $x$, the coefficients of the Poincar\'e meander
$$f(x,z)=\sum_{k \geq 1} \frac{x^k}{1+kz} $$
do not grow faster than $C^nn!$ (see Appendix A).  

%%%%%%%%%%%%%%%%%%%%%%%%%%%%%%%%%%%%%%%%%%%%%%%%%%%%%%%%%%%%%%%%%%%%
\subsection{Formal divergence}
We might be tempted to assert that the Poincar\'e meander series is necessarily divergent as the poles of partial sums accumulate at the origin. But one cannot, a priori, deduce divergence from the occurrence of poles of a partial sum of a series with rational functions as coefficients. One has to be careful. Let us for instance construct a non-trivial divergent expansion for the constant function equal to $1$. A simple induction shows that
$$1- \sum_{k = 0}^n \frac{(-1)^k k!}{\pi_k(z)} z^k=\frac{(-1)^{n+1}(n+1)!z^{n+1}}{\pi_{n+1}(z)} $$
where $\pi_k(z):=\prod_{j=1}^{k+1}(1-j z)$. In particular, one has the identity between formal power series:
$$1=\sum_{k \geq 0} \frac{(-1)^k k!}{\pi_k(z)} z^k $$
  and the partial sums of the right-hand side have poles that accumulate at the origin.
  
  More generally, given an arbitrary sequence of poles $\e_0,\e_1,\e_2,\ldots$, it is easy to construct an expansion
\[ c+\frac{a_0}{1-z/\e_0}+\frac{a_1z}{1-z/\e_1}+\frac{a_2z^2}{1-z/\e_2}+\ldots\]
which is in fact equal to $0$ in $\CM[[z]]$. After taking $a_0=-c$ to cancel the
constant, we just have to set $a_1=-a_0/\e_0$ to cancel the linear term,
then define $-a_2$ to be the coefficient of $z^2$ of the first two terms of the series and in general, if we set
$$f_n(z):=\sum_{k \leq n} a_k\frac{z^k}{1-\e_k^{-1}z}=b_{n+1}z^{n+1}+o(z^{n+1})  $$
then
$$f_{n+1}(z)=\sum_{k \leq n+1} a_k\frac{z^k}{1-\e_k^{-1}z}=o(z^{n+1}) $$
with $a_{n+1}=-b_{n+1}$. So the limiting series has a zero asymptotic expansion
at the origin. 
 %%%%%%%%%%%%%%%%%%%%%%%
 \subsection{Wolff's example}
 In order to deduce from the presence of poles that the sum is not a holomorphic function, we need the Borel property: the decomposition into polar parts must be unique.
 Here is an example, due to Wolff, where this is not the case~\cite{Wolff_1921}.
  
An {\em Apollonian gasket} is obtained by taking two tangent circles inside a big disc, also tangent to the boundary, and then considering
tritangent circles. We consider the following case:

\begin{figure}[htb!]
\centering
\input{Fig/Wolff_tikz}
\end{figure}

One shows that the union of the discs has full measure or, equivalently, that $\sum_kr_k^2=1$. An elementary computation shows that:
$$\frac1{2i\pi}\int_{D(a,r)}\frac{d\bar\xi\wedge d\xi}{z-\xi}=
\begin{cases}
\overline{z-a} &\text{for } z \in D(a,r),\\[4pt]
\dfrac{r^2}{z-a} &\text{for } z \notin D(a,r).
\end{cases}$$
%Indeed if $z \notin D(a,r)$,  using Stokes formula and the residue formula, we get that:
%\begin{align*}
%\frac1{2i\pi}\int_{D(a,r)}\frac{d\bar\xi\wedge d\xi}{z-\xi}
%&=\frac{1}{2i\pi}\int_{\d D(a,r)}
%\frac{\bar\xi\,d\xi}{z-\xi}\\
%&=\frac{r^2}{2i\pi}\int_{\d D(a,r)}
%\frac{d\xi}{(\xi-a)(z-\xi)}
%=\frac{r^2}{z-a}.
%\end{align*}
%If $z \in D(a,r)$, the differential form $\frac{d\bar\xi\wedge d\xi}{z-\xi}$ is singular at $\xi=z$. We denote by $\g_\e$ a small positively oriented loop around $z$ bounding a disc $\dt_\e$, so that $A_\e:=D(a,r) \setminus \dt_\e$ is an annulus. Stokes formula gives:
%\begin{align*}
%\frac1{2i\pi}\int_{D(a,r)}\frac{d\bar\xi\wedge d\xi}{z-\xi}
%&=\frac{1}{2i\pi}\int_{A_\e}
%\frac{d\bar\xi\wedge d\xi}{z-\xi}+\frac{1}{2i\pi}\int_{\dt_\e}\frac{d\bar\xi\wedge d\xi}{z-\xi}\\
%&=\frac{1}{2i\pi}\int_{\d D(a,r)}
%\frac{\bar\xi\,d\xi}{z-\xi}-\frac{1}{2i\pi}\int_{\g_\e}
%\frac{\bar\xi\,d\xi}{z-\xi}+\frac{1}{2i\pi}\int_{\dt_\e}\frac{d\bar\xi\wedge d\xi}{z-\xi}\\
%&=\frac{1}{2i\pi}\int_{\d D(a,r)}
%\Bigl(\bar a+\tfrac{r^2}{\xi-a}\Bigr)\tfrac{d\xi}{z-\xi}+\frac{1}{2i\pi}\int_{\g_\e}
%\frac{\bar\xi\,d\xi}{\xi-z}+\frac{1}{2i\pi}\int_{\dt_\e}\frac{d\bar\xi\wedge d\xi}{z-\xi}\\
%&=-\bar a+\frac{r^2}{z-a}\,(1-1)+\frac{1}{2i\pi}\int_{\g_\e}
%\frac{\bar\xi\,d\xi}{\xi-z}+\frac{1}{2i\pi}\int_{\dt_\e}\frac{d\bar\xi\wedge d\xi}{z-\xi}\\
%&\xrightarrow[\e\to0]{}-\bar a+\bar z+0=\overline{z-a}.
%\end{align*}
 Physically,  after complex conjugation, the integral gives the electric field created by a uniformly charged disc of radius $r$ with density $1/\pi$:
 $$ E=\overline{\frac1{2i\pi}\int_{D(a,r)}\frac{d\bar\xi\wedge d\xi}{z-\xi}}$$
  Outside the disc, it produces the same electric field as a point charge at the centre of the disc while inside the disc it is a linear function of the distance to the centre.

\begin{center}
\begin{tikzpicture}[>=stealth]
% field of the uniformly charged disc
\begin{scope}
\fill[gray!20] (0,0) circle (1.2);
\draw (0,0) circle (1.2);
\fill (0,0) circle (0.03);
\foreach \t in {0,45,...,315}
  {\draw[->] (\t:0.3) -- (\t:0.5);
   \draw[->] (\t:1.2) -- (\t:2.0);
   \draw[->] (\t:2.25) -- (\t:2.71);}
\foreach \t in {22.5,67.5,...,337.5}
  {\draw[->] (\t:0.7) -- (\t:1.17);
   \draw[->] (\t:1.7) -- (\t:2.26);}
\node at (0,-3.05) {$D(a,r)$};
\end{scope}
% modulus of the field
\begin{scope}[xshift=4.6cm,yshift=-1.2cm]
\draw[->] (0,0) -- (4.3,0) node[below] {$\rho=|z-a|$};
\draw[->] (0,0) -- (0,2.7) node[left] {$|E|$};
\draw[thick] (0,0) -- (1.2,2);
\draw[thick,domain=1.2:4,samples=60] plot (\x,{2.4/\x});
\draw[dashed] (1.2,0) node[below] {$r$} -- (1.2,2);
\draw[dashed] (0,2) node[left] {$r$} -- (1.2,2);
\node at (0.4,1.2) {$\rho$};
\node at (2.8,1.25) {$r^2/\rho$};
\end{scope}
\end{tikzpicture}
\end{center}

Hence, depending on the position of $z$ with respect to the different discs in the picture, the function $f(z)=\sum_{k\ge1}\frac{r_k^2}{z-a_k}$
takes different values. Outside the disc it is the sum of electric fields produced by the charges at each centre of the gasket, hence equal to the electric field of a point charge. Indeed for $|z|>1$, we have:
\begin{align*}
f(z)&=\sum_{k\ge1}\frac{r_k^2}{z-a_k}\\
&=\frac1{2i\pi} \sum_{k\ge1}\int_{D(a_k,r_k)}
  \frac{d\bar\xi \wedge d\xi}{z-\xi}\\
&=\frac1{2i\pi}\int_{D(0,1)}\frac{d\bar\xi \wedge d\xi}{z-\xi}
=\frac{1}{z} .
\end{align*}
This shows in particular that the function $f$ on $|z|>1$ has two different polar decompositions. Hence the polar decomposition is not always unique.

Inside the disc $|z|<1$ and outside the Apollonian discs, the electric field is identical to that of the disc again. Indeed, only the last step of the previous computation changes and gives:
$$f(z)=\frac1{2i\pi}\int_{D(0,1)}\frac{d\bar\xi \wedge d\xi}{z-\xi}=\bar z.$$
The sum is not even monogenic!

Inside a polar disc $D(a_j,r_j)$, all other discs act as point charges, so that the function becomes meromorphic with a simple pole:
\begin{align*}
f(z)
&=\frac{r_j^2}{z-a_j}+\sum_{k\neq j}\frac{r_k^2}{z-a_k}\\
&=\frac{r_j^2}{z-a_j}+\frac1{2i\pi} \sum_{k\neq j}\int_{D(a_k,r_k)}
  \frac{d\bar\xi \wedge d\xi}{z-\xi}\\
&=\frac{r_j^2}{z-a_j}+\frac1{2i\pi}\int_{D(0,1)}\frac{d\bar\xi \wedge d\xi}{z-\xi}
  -\frac1{2i\pi}\int_{D(a_j,r_j)}\frac{d\bar\xi \wedge d\xi}{z-\xi}\\
&=\frac{r_j^2}{z-a_j}+\bar z-\overline{z-a_j}\\
&=\frac{r_j^2}{z-a_j}+\bar a_j .
\end{align*}

This example shows that all kinds of strange phenomena may appear in general.

 %%%%%%%%%%%%%%%%%%%%%%%%%%%%%%%%%%%%%
\subsection{ The quasi-analytic threshold}
\label{SS::flat}
 Let us now give a family
 of examples which will give us an idea of the quasi-analytic threshold.
  For any $\alpha>1$, we consider the product
\[
G_\alpha(u):=\prod_{n\ge1}
\left(1-\frac{u^2}{n^{2\alpha}}\right).
\]
It defines an entire function which
grows faster than any power of $|u|$ as
$ u\to\infty$ in any closed sector which does not meet the real axis.
Hence, the function
\[
f(z):=\frac{1}{G_\alpha(z^{-1})}
=\prod_{n\ge1}\left(1-\frac{1}{z^2n^{2\alpha}}\right)^{-1}
\]
is flat at the origin in any such sector. Its poles are the points $\pm n^{-\alpha}$, which accumulate at the origin from both sides.
We decompose $1/G_\alpha$ into simple elements and replace $u$ by $1/z$:
\begin{align*}
f(z)
&=\sum_{n\ge1}\frac{1}{G_\alpha'(n^\alpha)}\left(\frac{1}{z^{-1}-n^{\alpha}}-\frac{1}{z^{-1}+n^{\alpha}}\right)\\
&=z\sum_{n\ge1}c_n\left(\frac{1}{z-n^{-\alpha}}+\frac{1}{z+n^{-\alpha}}\right),
\qquad c_n:=-\frac{1}{n^{\alpha}G_\alpha'(n^\alpha)} .
\end{align*}
An explicit computation gives:
$$
c_n=\frac12\prod_{\substack{m\ge1\\m\ne n}}
\left(1-\frac{n^{2\alpha}}{m^{2\alpha}}\right)^{-1} \sim
(-1)^{n+1}\alpha(2\pi)^{\alpha-1}n^{\alpha-1}
\exp\left(-\pi n\cot\frac{\pi}{2\alpha}\right).
$$
Since $\cot\frac{\pi}{2\alpha}>0$ for $\alpha>1$,
 the coefficients $c_n$ decay exponentially. In particular, spaces of series of the form
$$ \sum_{n\ge1}c_n\left(\frac{1}{z^{-1}-n^{\alpha}}-\frac{1}{z^{-1}+n^{\alpha}}\right),\ c_n=O(q^n),\ e^{-\pi\cot\frac{\pi}{2\alpha}}<q<1 $$
are not quasi-analytic.
%%%%%%%%%%%%%%%%%%%%%%%%%%%%%
 \subsection{Diophantine approximation}
 The previous example can be adapted to Diophantine approximation. It shows that
 in dynamical systems and in KAM theory, monogenic expansions discovered by Arnold in~\cite{Arnold_SD1}
 are above the quasi-analytic threshold: the corresponding spaces of monogenic functions are not quasi-analytic. This is in accordance with a question raised by Herman, who asked whether the solution of the conjugacy problem is determined by its Taylor series at a Diophantine point and expected a negative answer~\cite[III.16]{Herman1985}.
 
By Dirichlet's theorem, any irrational number $x$ admits infinitely many rational approximations with
$$\left|x-\frac{p}{q}\right|<\frac{1}{q^2} .$$
So if we remove from the real line an interval of radius $c/q^\alpha$ around each rational number $p/q$, nothing remains for $\alpha<2$.
For $\alpha>2$, the remaining points form a set
$$\Omega_{c,\a}:=\left\{ x \in \RM: \left|x-\frac{p}{q}\right| \geq \frac{c}{q^\a} \text{ for all rational numbers } \frac{p}{q}\right\},$$
and the union
$$\Omega_\a:=\bigcup_{c>0} \Omega_{c,\a}$$
 has full measure by the Borel--Cantelli theorem. At the threshold $\alpha=2$, the set $\Omega_\a$ has measure zero, but its Hausdorff dimension is equal to $1$ (Jarn\'ik's theorem~\cite{Jarnik_1928}). In all these cases, the rational numbers accumulate at the point $x$ from both sides. Let us construct a sum of polar parts which is flat at such a point without being zero.

So we take a point $x \in \Omega_{c,\a}$ and construct a function flat at $x$. For each $n\ge1$, we choose two rational numbers $x_n^+$ and $x_n^-$ such that
$$\left|x_n^\pm-\left(x\pm n^{-\alpha}\right)\right|\le \frac{1}{4n^{\alpha+2}} ,$$
and we replace the zeros $\pm n^{\alpha}$ of the function $G_\alpha$ constructed previously by the real numbers
$$u_n^\pm:=\frac{1}{x_n^\pm-x} .$$
We consider the functions
$$G(u):=\prod_{n\ge1}\left(1-\frac{u}{u_n^+}\right)\left(1-\frac{u}{u_n^-}\right),\qquad F(z):=\frac{1}{G\left((z-x)^{-1}\right)} .$$
The poles of $F$ are the rational numbers $x_n^\pm$, which accumulate at the point $x$ from both sides. As before, one can check that the function $F$ is flat at the point  $x$.  The Dirichlet bound is above the threshold.

%%%%%%%%%%%%%%%%%%%%%%%%%%%%%
\section{Stacks and perforations}
\subsection{  The stacked space $\X$}
We will be dealing with sequences 
$$f_1,f_2,\ldots,f_n,\ldots$$ of holomorphic functions 
$f_n:U_n \to \C,$ 
where  $U_n \subset \CM$ are open subsets. As our objective is the construction of a sheaf, one may consider the projective line 
$\PM=\CM \cup \{ \infty \}$ instead of $\CM$ in what follows.

Our aim is very classical: study the set of $z \in \bigcap_{n=1}^\infty U_n$ for which the series $\sum_{n=1}^\infty f_n(z)$ converges, and investigate the nature of the sum-function so defined. 

Let $\Nt$ be a discrete countable topological space. In most cases $\Nt=\NM$ but not always.
In order to study this in a systematic way, we form the product
space
\[\X:=\Nt \times \CM,  \]
and consider it as a one-dimensional complex manifold. It comes equipped with two projection maps
$$\nu:\X \to \Nt,\;\; (n,z) \mapsto n,\;\;\;\;\pi:\X \to \CM,\;\; (n,z)  \mapsto z.$$

%producing a diagram 
%\[   \begin{array}{rcc}
%        \X &\stackrel{q}{\longrightarrow}& \N \\
%        p \downarrow& &\\
%        \P& &\\
%     \end{array}
%\]
We may envision $\X$ as an infinite stack\footnote{The word {\em stack} is usually used in algebraic geometry for a much more elaborate mathematical structure. Our use
     of the word is more in line with
     that in common language, like a stack of paper sheets.} of sheets
$\X_n:=\nu^{-1}(n)$. Via the map $\pi$, the sheets are identified with the complex line $\CM$ and are therefore equipped with a flat metric. Such stacks form a trivial example of a log-Riemann surface~\cite{Perez-Marco_Biswas}.  Note that any subset $\cU \subset \X$ automatically comes with a similar stacked structure, by restricting the maps $\pi$ and $\nu$ to $\cU$. So, an
open subset $\cU \subset \X$ is the disjoint union of open sets
$\cU_n=\cU \bigcap \X_n$ which, via $\pi$, corresponds 1-1 to sequences of open
sets $U_n=\pi(\cU_n) \subset \CM$, and $\cU_n=\pi^{-1}(U_n) \cap \X_n$. Similarly, a
sequence of holomorphic functions $f_n: U_n \to \C$  corresponds 1-1 to a
holomorphic function $f: \cU \to \C$, defined by putting $f(n,z)=f_n(z)$, etc.
\begin{center}
  \includegraphics[width=0.4\linewidth]{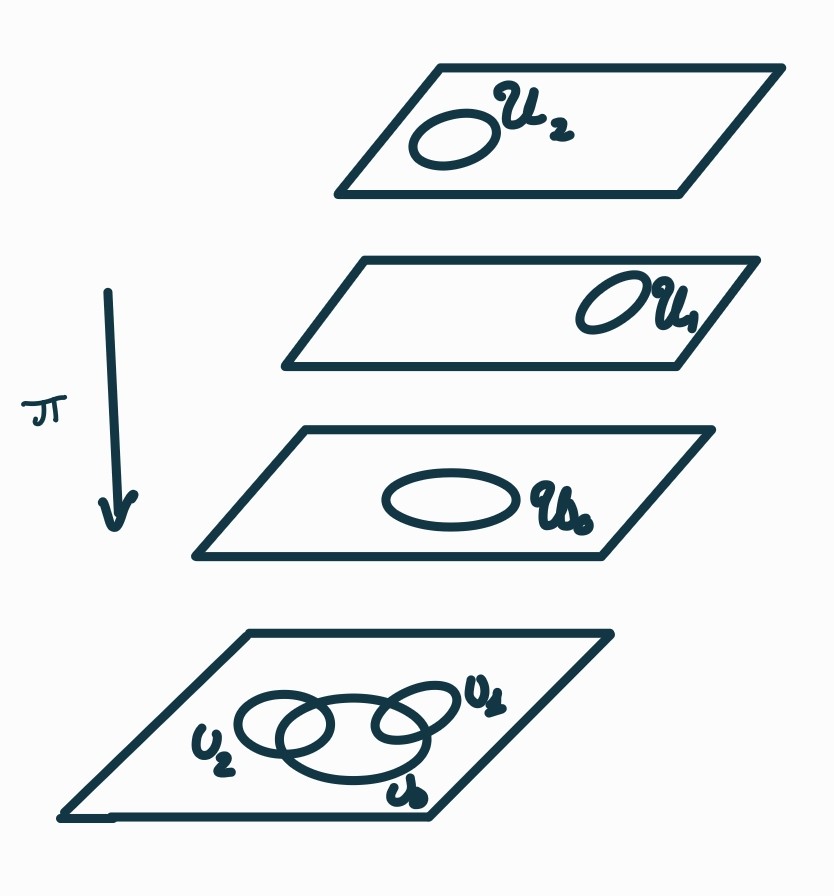} 
\end{center}

%More generally, we lift usual functions and operations to $\X$ and use notation%s:
%$$e^z,\ z^\a,\ \ zz',\ z+z',\ \text{etc.},\ z=(n,w),\ z=(n,w')$$
%for
%$$(n,e^w),\ (n,w^\a),\ (n,\frac{1}{ww'}),\ (n,ww'),\ (n,w+w'),\ \text{etc.} $$

Typically we are interested in sequences $f_1, f_2,\ldots$ of rational functions, or more generally functions with isolated singularities, like meromorphic functions, where singularity always means an unramified singularity.
%A typical example is given by the Poincar\'e meander:
%$$f(x,z)=\sum_{n \geq 0} \frac{x^k}{1+nz} $$
%In our setting $x$ is fixed inside the unit disc and we consider the holomorphi%c functions:
%$$f_n:\X \setminus \Pt \to \C,\ (n,z) \mapsto  \frac{x^n}{1+nz} $$
%where
%$\Pt=\{ (n,z): 1+nz=0 \}$

In this language, we are just dealing with a holomorphic function on an open subset of $\X \setminus \cP$, where $\cP \subset \X $ is a {\em discrete} set of singularities. The set $P_n:=\pi(\cP \bigcap \X_n)$ then is the set of singularities of the function $f_n$ and so $f_n$ is holomorphic on $U_n \subset \pi((\X \setminus \cP)_n)=\CM \setminus P_n$. If the function $f_n$ is unbounded at a point of $P_n$, that is if the singularity is non-removable
then we say that the point is an {\em isolated singularity}\footnote{We will use the word isolated singularity for any kind of unramified non-removable isolated singularity of a holomorphic function and refer to an essential singularity as a singularity of infinite order.} of $f$. 
The union of the singularity sets $P_n:=\pi(\cP_n) \subset \CM$ is precisely the image set $P:=\pi(\cP)$. This set is countable, and it may have points of accumulation or even be dense in $\P$. However, by the discreteness of $\cP$ in $\X$, for $z_0 \notin P$ we always find a neighbourhood of the fibre $\pi^{-1}(z_0)$ disjoint from $\cP$. This is one of the technical advantages of working on $\X$, rather than on $\C$ or $\P$.

Our sequence can be defined sheaf theoretically in a classical way: the holomorphic functions $f=(f_n)$ on open subsets of $\XM$ are sections of the sheaf
$\Ot_{\X}(*\cP)$ on $\X$, where the $*$ in front of $\cP$ indicates that
arbitrary pole orders are allowed. By definition, a section $f$ of the direct
image sheaf $ \pi_*\Ot_{\X}(*\cP) $
over an open set $V \subset \CM $ is the same as a section of $\Ot_{\X}(*\cP)$ over $\pi^{-1}(V)$ and thus can be identified with a sequence $f_1,f_2,\ldots$, where $f_n$ is holomorphic on $V \setminus \Pt_n$.  

%%%%%%%%%%%%%%%%%%%%%%%
\subsection{   Perforations}
We will use the notation $D(\omega,r) \subset \XM$ for the open disc of radius $r$ centred at the point $\omega=(n,z) \in \X$.

\begin{definition} Let $\cP \subset \XM$ be a discrete subset and  $r:\cP \to \RM_{>0}  $ a bounded function. 
\begin{enumerate}
 \item We call $r$ a  {\em   radius function}  if
for any $\omega \neq \eta \in \Pt$ the closures of the discs $D(\omega, r),D(\eta, r) \subset \X$ do not intersect
 and if $r_n:=\inf_{\omega \in \Pt_n} r(\omega)>0$.
  \item We call $r:\cP \to \RM_{>0}  $ a  {\em  $\l$-radius function}, $\l >1$, if $\l r$ is a radius function.
    \item Given a radius function $r$, the {\em $r$-perforation} of a set $\Vt \subset \XM$ is the subset $\Vt(r) \subset \XM$ obtained by removing the corresponding discs: 
$$\Vt(r):=\Vt \setminus \bigcup_{\omega \in \Pt} \left(\Vt \cap D(\omega,r(\omega)) \right) $$ 
\item Given a radius function $r$, the {\em $r$-perforation} of a set $V \subset \CM$ denoted  $\Vt(r) \subset \XM$  is the perforation of $\pi^{-1}(V)$
\item The {\em residual set} of $V$ with respect to $r$ is defined by:
$$ \Rt_r(V):=\bigcap_n \pi(\Vt(r)_n)$$

\end{enumerate}
\end{definition}
We often omit the subscript $r$ in the definition of the residual set.
Note that the set  $ \Rt(V)$ is closed in $V$ and has some resemblance to a {\em Swiss cheese}, as it is obtained by
removing projected discs centred at $\omega$ of radius $r(\omega)$. In general it will be constructed similarly to a Cantor set,
where we remove open discs instead of intervals. It is our philosophy to avoid as much as possible the
considerations on this cheese, and rather work with the much better behaved set $\Vt(r)$. The residual set will play
the role of a convergence domain as we shall now see. In practical applications, the centre of the disc is not the pole of the function as a disc contains
a collection of several poles, see for instance~\cite[\S 7]{Arnold_SD1}.

An important case of radius function is given by Arnold functions.
\begin{definition} A  {\em $\l$-Arnold function} or simply an {\em Arnold function} is a $\l$-radius function such that for $\pi(\omega) \neq \pi(\eta)$, the
    projected discs $\pi(D(\omega,\l r)), \pi(D(\eta,\l r)) $ do not intersect.
\end{definition}

%%%%%%%%%%%%%%%
\subsection{  Holomorphic sequences}

The perforation $\X(r) \subset \X$ is obtained by removing open discs from $\X$. Let $V \subset \CM$ be an open subset and
$$f=(f_n): \Vt(r) \to \CM $$ a holomorphic function.  Note that $\Vt(r)$ is, in general, not an open subset: we took away open discs so it has boundary points. To be holomorphic in $\Vt(r)$ at the boundary points means, as usual, to be holomorphic in a larger open set containing $\Vt(r)$.  However it may happen that the function cannot be extended meromorphically inside the disc $D(\omega,r(\omega))$; it may also happen that such an extension exists, but the singularity is not the centre of the disc.

The functions $f_n$ have the residual set $\Rt(V)=\bigcap_n \pi(\Vt(r)_n)$
as common definition domain.  One can form the set
\[\Rt(f) :=\{z_0 \in \Rt(V) \;|\;\sum_n |f_n(z_0)| < \infty\}\]
of points where absolute convergence takes place. On this set we have a
well-defined sum-function
\[ S(f): \Rt(f) \to \C,\;\;\; z_0 \mapsto S(f)(z_0):= \sum_n f_n(z_0) .\]
However, in order to say anything interesting about the sum function $S(f)$ we need to have some property
of uniformity of the convergence.  

This construction has a sheaf theoretic interpretation.  We consider the restriction $\Ot_{\XM|\XM(r)}$ of the sheaf $\Ot_{\XM}$ to $\XM(r)$. The  direct image sheaf   $ \pi_*\Ot_{\XM|\XM(r)}$
is the {\em sheaf of holomorphic sequences in the $r$-perforation}. 
A section of this sheaf over $V \subset \CM$ is the same as a holomorphic
function $f$ on $\Vt(r)$ and therefore the same as a sequence $(f_n)$. %%%%%%%%%%%%%%%%%%%%%%%%%%%%%%%%%%%%%%%%%%%%%%%%%%%%%%%%%%%%%%%%%%%%%%%%%%%%%%%%%%%%%%%%%%%%%%%%%%%%%%%%
\subsection{  Summable holomorphic sequences}
We now introduce a strong condition of uniformity of convergence
that leads to regularity properties of the summation. We fix a radius function $r:\Pt \to \RM_{>0}$.

For a set $V \subset \CM$ and a bounded function $f$ on $\Vt(r)$, we call
\[ |f|_{V}:=\big(\sup_{z \in \Vt(r)_n} |f_n(z)|, n \in \Nt\big)=\big(\|f_n\|_{C^0(\Vt(r),\CM)}\big)= \big(|f_n|_{\Vt(r)_n}\big)\]
the {\em norm sequence of $f$ on $V$ in the $r$-perforation.}
Note that if $V$ is compact, the sets $\Vt(r)_n$ are compact as well, so for
any holomorphic function $f \in \Ot_{\XM}(\cU)$ defined on an open subset
$\cU \supset \Vt(r)$ the norm sequence $|f|_{V}$ is defined.

\begin{definition} Let $\cU \subset \XM$ and $\Vt(r) \subset \cU$.
  A holomorphic function $f \in \Ot_{\XM}(\cU)$ is called {\em $r$-summable on $V$},
  if its norm sequence $|f|_{V}$ is summable:
\[ \sum_{n \in \Nt} \sup_{z \in \Vt(r)_n}|f_n(z)| <\infty.\]
\end{definition}  

\begin{definition} Let $V \subset \CM$ be an open subset.\\
1) An element 
$f \in \pi_*\Ot_{\XM|\XM(r)} (V)$ is called {\em locally $r$-summable}, if each point $z_0 \in V$ has a neighbourhood $V' \subset V$ so that $f$ is $r$-summable in $V'$.\\
2) We set  
\[\Ot_{\CM}^r(V):=\{ f \in  \pi_*\Ot_{\XM|\XM(r)} (V)  \;\;|\;\;\textup{$f$ is locally $r$-summable}\}\]
\end{definition}
  For $f \in \Ot^r_{\CM}(V)$, the sum function $F=S(f)$ is the limit over finite subsets of $\Nt$; it is therefore continuous on the residual set $\Rt_r(V)$.
Note that locally $r$-summable functions define a subsheaf of the sheaf of holomorphic sequences $\Ot_{\CM}^r \subset  \pi_*\Ot_{\XM|\XM(r)} $
of {\em locally $r$-summable sequences}. 

 %%%%%%%%%%%%%%%%%%%%%%%%%%%%%%%%%%%%%%%%%%%%%
\subsection{   Meandromorphicity}
 %%%%%%%%%%%%%%%%%%%%%%%%%%%%%%%%%%%
We now assume that $\Nt=\NM$ and therefore $\XM=\NM \times \CM$.
 \begin{definition} A section $f \in \Ot_\CM^r(V)$ is called {\em meandromorphic} if every point of $V$ has a neighbourhood $V'\subset V$ and constants $A>0$, $0<\rho<1$ such that
\begin{enumerate}
\item $|f_n|_{\Vt'(r)}\le A\rho^n$ for all $n$;
\item the numbers $N_n(V'):=\#\{\omega\in\Pt_n:\pi(\omega)\in V'\}$ are finite and $\sum_nN_n(V')\rho^n<\infty$.
\end{enumerate}
\end{definition}
Meandromorphic sequences form a subsheaf $\Mt^r \subset \Ot^r_\CM$.

If $f=(f_n	)$ is a sequence of functions on subsets $(\Vt(r)_n)$ then their sum is called a {\em meandromorphic sum}. From a sheaf theoretical point of view such sums define a sheaf $S(\cO_{ \CM|\Rt(\Vt(r))}^r)$ and
the operation of summation defines a morphism of sheaves:
\[ S: \Ot_{\CM|\Rt(\Vt(r)) }^r \to S(\cO_{ \CM|\Rt(\Vt(r))}^r),\;\;f \mapsto S(f).\]

The definition of meandromorphicity can be understood as follows.
Consider a series in two variables
$$F(x,z):=\sum_{n \geq 0} a_n(z) x^n,\;\;\;f = (f_n) \in \Ot_\X(\Vt(r)) $$
as we did in the first section. If the sequence of norms $(| a_n |_{\Vt(r)_n})$ is bounded, then the function $F$ is a meandromorphic sum
for fixed values of $x$, provided that $|x|$ is small enough and that the number $N_n$ of poles in the $n$-th sheet grows at most exponentially.

%%%%%%%%%%%%%%%%%%%%%%%%%%%%%%%%%

\section{Holomorphicity in the residual set}
We will now give a criterion for a meandromorphic sum to be $C^k$-holomorphic in the sense of Borel and Whitney. 
%%%%%%%%%%%%%%%%%%%%%%%%%%%%%%%%%%%%%%%%%%%%%
\subsection{ The Cauchy estimates}
 We will bound derivatives with the Cauchy inequality,
$$|f'(z)| \leq \frac{1}{\rho}\sup_{w \in D(z,\rho)}|f(w)|,  $$
which is a simple consequence of the Cauchy integral formula.
(Here as usual $D(z,\rho)$ denotes the open disc centred at $z$ of
radius $\rho$). Note that application of this estimate requires a
shrinking of the domains of definition.

\begin{definition} Let $U \subset V \subset \CM$. The {\em Huygens separation}
 $\dt(U,V)$ is defined by
 $$\sup \{ \rho \in \RM: \forall z \in U, D(z,\rho) \subset V \}. $$
% $$|f'|_U \leq  \frac{1}{\dt(U,V)}|f|_V$$
% where the norms $|\cdot|_U$, $|\cdot|_V$ are the supremum norms.
\end{definition}
\begin{center}
\includegraphics[width=0.4\linewidth]{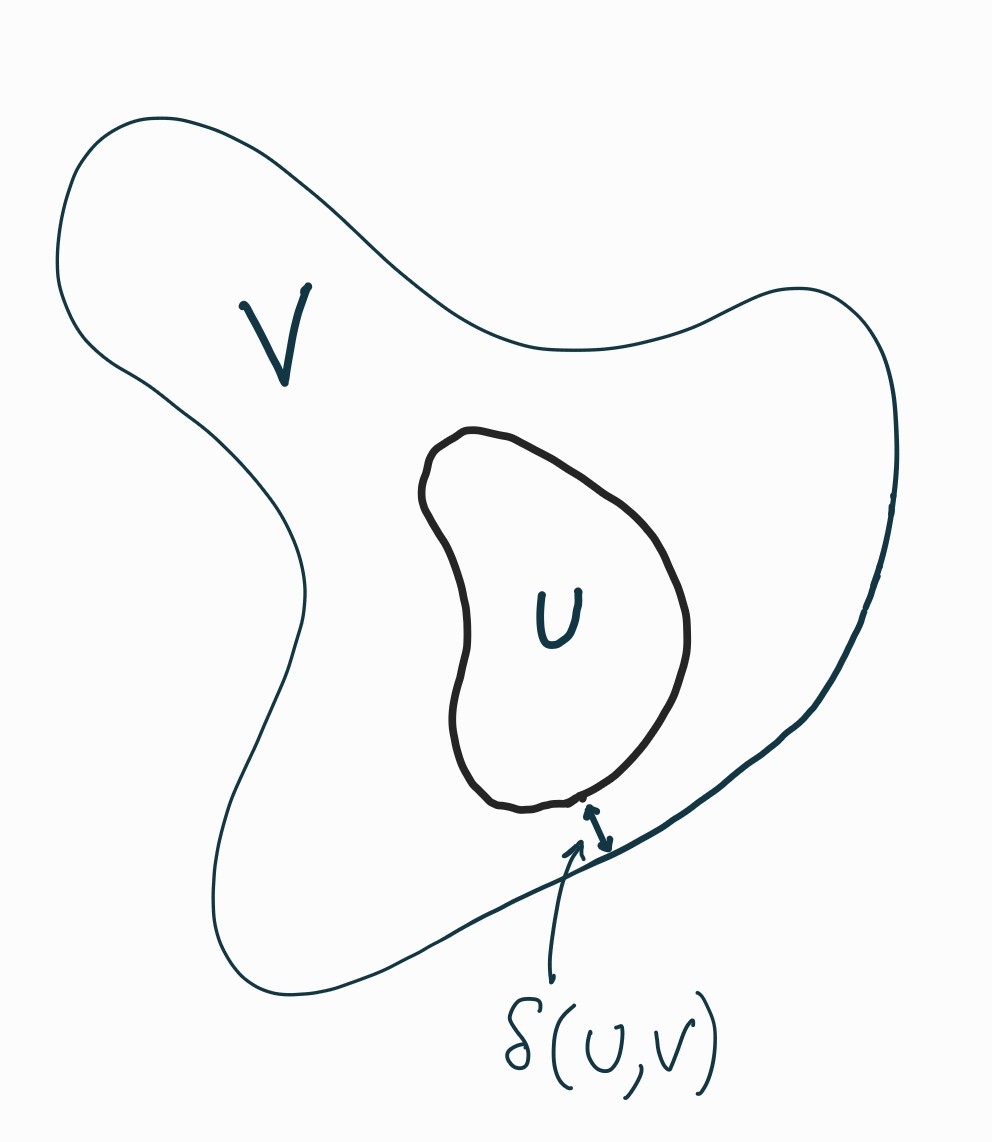}
\end{center}
The Cauchy integral formula implies that  its $k$-th derivative satisfies the estimate
$$f^{(k)}(z)=\frac{k!}{2i\pi} \int_\g \frac{f(w)}{(w-z)^{k+1}}dw .$$ 
Hence if  $U \subset V \subset \CM$ are open subsets such that $\dt(U,V)>0$ and if the function $f$ is a bounded holomorphic
function on $V$ then:
$$|f^{(k)}|_U \leq \frac{k!}{\dt(U,V)^k} |f|_V$$
where $|\cdot|_U$ and $|\cdot|_V$ denote the supremum norms in $U$ and $V$. 
%%%%%%%%%%%%%%%%%%%%%
\subsection{ The shrinking lemma}
If $U \subset V$ are two sets and $r$ is a radius function,  the Huygens
separation $\dt(\Ut(r),\Vt(r))$ will be strictly smaller than $\dt(U,V)$ so it could decrease when going to higher sheets.  
However if $r$ is a $\l$-radius function, one has an explicit lower bound for $\dt(\Ut(\l r),\Vt(r))$ using $\dt(U , V)$, namely:

\begin{lemma}
\label{L::Huygens}
Let $r$ be a $\l$-radius function such that $r\le 1$ and let $r' =\l  r$. If two sets $U \subset V \subset \CM$ satisfy
 $$\dt(U,V) \geq \e  >0,\ \e=\l-1, $$
then
$$\dt(\Ut(r')_n, \Vt(r)_n) \geq \e r_n\ \text{ for all } n,\qquad r_n=\inf_{\omega\in\Pt_n}r(\omega) .$$
\end{lemma}
\begin{proof}
Take $ z \in \Ut(r')_n$ and $ x \in D(0,\e r_n) \subset D(0,\e)$, as $r \leq 1$. As  by assumption $\dt(U,V) \geq \e$, this implies $z+x \in V$.
Moreover for any $\omega=(n,p) \in \Pt$, we also have
\begin{align*}
|z+x-p| &  \geq |z-p|-|x| \\
  & \geq (1+\e) r(\omega)- \e r_n\geq r(\omega) 
  \end{align*}
  and therefore  $z+x \notin D(p,r(\omega))$. This shows that $z+x \in  \Vt(r)_n$,
  in other words, we have $\dt(\Ut(r')_n,\Vt(r)_n)\ge \e r_n$.
\end{proof}

For a disc $U=D(z,r)$, we use the notation $\l U$ for the disc $D(z,\l r)$.
\begin{corollary}
\label{C::Huygens}
Let $r\le1$ be a $\l$-radius function, $\e=\l-1$, and let $U$ be the open unit disc. Then for every pole $\omega\in\Pt$ with $\pi(\omega)\in U$, one has $D(\omega,\e r(\omega))\subset\l U$.
\end{corollary}
\begin{proof}
For $z\in D(\omega,\e r(\omega))$, $|z|<|\pi(\omega)|+\e r(\omega)<1+\e=\l$.
\end{proof}

%%%%%%%%%%%%%%%%%%%%%%
\subsection{  The Taylor lemma}
Borel's $C^k$-monogenic functions which are the same as $C^k$-Whitney holomorphic functions are defined as follows:
$f$ is $C^k$-monogenic in $X$ if there exist continuous functions $a_0=f,a_1,\dots,a_k$ such that, for $j=0,\dots,k$,
$$a_j(y)=\sum_{i=0}^{k-j}\binom{j+i}{j}a_{j+i}(x)(y-x)^i+|y-x|^{k-j}R_j(x,y)  ,$$
with $R_j(x,y) \xrightarrow[x,y \to a]{}0$. The functions $j!\,a_j$ are then called the {\em holomorphic Whitney derivatives of $f$} or simply the  {\em derivatives of $f$}. 
 
Let  $V \subset \CM$ be an open subset. Any element $f \in \Ot^r(V)$ is a holomorphic function on a neighbourhood of the disjoint union of the $\Vt(r)_n$'s. Hence, like any holomorphic function, $f$ can be differentiated.
The following lemma shows that, under a stronger summability condition,
differentiation and summation commute:

\begin{lemma}
\label{L::Taylor}
Let $r\le1$ be a $\l$-radius function, $r'=\l r$, $\e=\l-1$, $f\in\Ot^r(V)$, and assume that
\[ C_m:=\sum_{n} \frac{|f|_{\Vt(r)_n}}{r_n^m} < \infty.\]
Let $V'\subset V$ be a convex open set with $\dt(V',V)\ge\e$, and assume that $(1+2\e)\,r(\omega)\le\e$ for every pole $\omega$ whose disc $D(\omega,\l r(\omega))$ meets $V'$. Then for all $z,z_0\in\Rt_{r'}(V')$ and $k\le m$:
\[ |S(f^{(k)})(z_0)|\le C_k\frac{k!}{\e^k},\qquad
\Bigl|S(f)(z)-\sum_{k=0}^{m-1}\frac1{k!}S(f^{(k)})(z_0)(z-z_0)^k\Bigr|\le \Bigl(\frac\pi2\Bigr)^m\frac{C_m}{\e^m}|z-z_0|^m .\]
In particular $S(f)$ is $C^{m-1}$-holomorphic in the sense of Whitney on $\Rt_{r'}(V')$; it is $C^m$ if moreover $C_{m+1}<\infty$.
\end{lemma}
\begin{proof}
Write $|\cdot|_n$ for the supremum on $\Vt(r)_n$ and $\rho_n=\e r_n$. By the shrinking lemma (Lemma~\ref{L::Huygens}), $D(z_0,\rho_n)\subset\Vt(r)_n$ for $z_0\in\Vt'(r')_n$, and  Cauchy's inequalities give $|f_n^{(k)}(z_0)|\le k!\rho_n^{-k}|f|_n$, whence the first bound, since $C_k\le C_m$ for $k\le m$.

By convexity, the segment $[z_0,z]$ lies in $V'$. The preimage of the segment $[z_0,z]$ under the projection $\X \to \CM$ is a disjoint union of segments, one on each sheet. Each of these segments
 may cross some discs $D(\omega,\l r(\omega))$. If this is the case, we replace (differently on each sheet) the chord of the segment inside such a disc by the shorter arc of its boundary circle. Since the closures of the $\l r$-discs are disjoint on each sheet, this gives on each sheet a path $\g$ from $(n,z_0)$ to $(n,z)$ of length $L\le\frac\pi2|z-z_0|$, where $n$ denotes the height of the sheet.

Let us show that $D(w,\rho_n)\subset\Vt(r)_n$ for every point $w\in\g$. The path $\g$ lies outside all the $\l r$-discs of the sheet, so, as in the proof of the shrinking lemma, the disc $D(w,\rho_n)$ meets none of the discs $D(\omega',r(\omega'))$. Moreover, a point $w\in\g$ is at distance at most $\l r(\omega)$ from a point of the segment, so
$$\mathrm{dist}(w,\partial V)\ge\e-\l r(\omega)\ge\e r(\omega)\ge\rho_n.$$
Hence the Cauchy inequalities give $|f_n^{(m)}(w)|\le m!\,\rho_n^{-m}|f|_n$ for every $w\in\g$.
After parametrisation of the path by the remaining arclength $s$, we have $|z-w|\le s$. We use this estimate to bound the remainder in the Taylor formula:
\begin{align*}
f_n(z)-\sum_{k<m}\frac1{k!}f_n^{(k)}(z_0)(z-z_0)^k&=\int_\g f_n^{(m)}(w)\frac{(z-w)^{m-1}}{(m-1)!}\,dw,\\
\left|f_n(z)-\sum_{k<m}\frac1{k!}f_n^{(k)}(z_0)(z-z_0)^k\right| & \leq \sup_{w\in\g}|f_n^{(m)}(w)|\int_0^L\frac{s^{m-1}}{(m-1)!}\,ds\\
 & = \sup_{w\in\g}|f_n^{(m)}(w)|\ \frac{L^m}{m!}\\
 & \leq |f|_n\ \left(\frac{\pi |z-z_0|}{2\rho_n}\right)^m .
\end{align*}
 Summing over $n$ gives the bound. Applying it to $f^{(k)}$, along the same paths, gives the  Whitney conditions for the 
 derivatives.
\end{proof}

%%%%%%%%%%%%%%%%%%%%%%%%%%%%%%%%%%%%%%%%%%%%%%%%%%%%%%%%%%%%%%%%%%%%%%%%%%%%%%
\section{Polar decomposition}
%%%%%%%%%%%%%%%%%%%%%%%%%%%%
\label{S::polar}
\subsection{  Decomposition of $\X$}
Clearly, our set-up has various generalisations. 
The complex line $\CM$ may be replaced by a general Riemann surface; and finally,
one might try to replace the stacked situation $\X \lra \CM$ by a general holomorphic
map between (disconnected) Riemann surfaces. We leave it to the reader to contemplate the details and refer to \cite{Perez-Marco_Biswas} for inspiration.

Here we will concentrate on a version of the decomposition of a meromorphic
function into a polar part and a holomorphic part. For this we first decompose the space itself. If $(\cP, r)$ is a perforation
of $\X$, we can consider 
$$\X[\cP]:=\cP \times \CM,$$ the space with sheets $\X_p$, labelled
by the singularities $p=(n, z_p) \in \cP$. On $\X[\cP]$ we have an induced perforation, with polar
set
$$\cP' =\{(p,\pi(p)) \in \cP \times \CM\;|\;p \in \cP\}.$$
 The set $\cP' \subset \X[\cP]$ is thereby mapped bijectively to $\cP \subset \X$ and therefore there is an induced radius function on $\cP'$ that we also denote by $r$.  
Note that a meromorphic function $g$ on $\X[\Pt]$ with singularities in $\Pt'$, i.e. an element
of $\Ot_{\X[\Pt]}(*\Pt')$, has at most a single singularity on each of its sheets.
 
   Furthermore, there is a canonical map
  $$\p:\Pt' \to \NM, (p,\pi(p)) \mapsto n,\ p=(n,\omega) $$
  which indicates the sheet to which a singularity belongs.
\begin{definition}
A holomorphic function $g=(g_p)_{p\in\Pt}$ on the $r$-perforation of $\X[\Pt]$ is called {\em meandromorphic} if every point of $\CM$ has a neighbourhood $V'$ and constants $A>0$, $0<\rho<1$ such that
\begin{enumerate}
\item $|g_p|_{\Vt'(r)}\le A\rho^{\p(p)}$ for every $p\in\Pt$;
\item the sets $\Pt_n$ are finite and $\sum_n N_n\rho^n<\infty$, where $N_n=\#\Pt_n$.
\end{enumerate}
\end{definition}
 
%%%%%%%%%%%%%%%%%%%%%%%%%%%%%%%%
 \subsection{  Polar part at a disc}
 \label{SS::polar}
 Our first observation belongs to the theory of meromorphic functions in one variable\footnote{Our original argument has been simplified by P\'erez-Marco.}.

An ordinary holomorphic function on an annulus admits a unique Laurent expansion at a singularity $\omega$. Indeed consider two discs $D  \subset D'$ with the same centre $\omega \in \CM$
having radii $r,\ \l r$ and consider the annulus (to simplify the notations we take $\omega=0$):
\[A:=A(\l,r):=\{ z \in \C: r < |z| < \l r \} \subset \C,\ \l>1\]
Then define
$$a_n=\frac{1}{2i\pi}\int_\g f(z)z^{-n-1}dz $$
where $\g$ is a circle inside the annulus, and
$$ f_\omega(z)=\sum_{n  \leq -1}a_n z^n,\ h_\omega(z)=\sum_{n  \geq 0} a_n z^n.$$
For $\g$ the circle of radius $\rho\in(r,\l r)$ one has $|a_n|\le|f|_A\rho^{-n}$; letting $\rho\to r$ for $n\le-1$ and $\rho\to\l r$ for $n\ge0$,
$$|a_{-k}|\le|f|_A\, r^{k}\ (k\ge1),\qquad |a_n|\le|f|_A(\l r)^{-n}\ (n\ge0).$$
Hence $f_\omega$ is holomorphic for $|z|>r$ and $h_\omega$ is holomorphic for $|z|<\l r$. We call $f_\omega$ the polar part of $f$ at $\omega$. Moreover, for $|z|\ge\mu r$, $\mu>1$,
$$|f_\omega(z)|\le|f|_A\sum_{k\ge1}\Bigl(\frac r{|z|}\Bigr)^k\le\frac{|f|_A}{\mu-1}.$$

\subsection{ Polar decomposition}
 In classical complex function theory, the difference between a meromorphic function and the sum of its polar parts extends holomorphically (Mittag-Leffler theory). Our situation is similar.
\begin{proposition}
\label{P::polar}
Let $r\le\e$ be a $\l$-radius function, $\e=\l-1<1$, $0<\dt<\e$, let $U$ be the open unit disc, and let $f \in \Mt^{\dt r}(\overline{\l U})$. Then the polar parts $f_\omega$, $\pi(\omega)\in U$, define a meandromorphic function $P(f)$ on the Riemann sphere $\PM$ perforated at the discs $D(\omega,r(\omega))$, $\pi(\omega)\in U$, and $S(f)-S(P(f))$ extends to a holomorphic function on $(1-\e)U$.
\end{proposition}
\begin{proof}
Let $A,\rho$ be such that $|f|_n \leq A\rho^n$ where the supremum norm is taken on $\l U(\dt r)$ and define $N_n:=\#\{\omega\in\Pt_n:\pi(\omega)\in U\}$.  For $\omega\in\Pt_n$ with $\pi(\omega)\in U$, the annulus $A_\omega=D(\omega,\e r(\omega))\setminus\overline D(\omega,\dt r(\omega))$ lies in $\l U$ (Corollary~\ref{C::Huygens}); so $f_n$ is holomorphic on it. By the estimates of \S\ref{SS::polar} with inner radius $\dt r(\omega)$, for $|z-\pi(\omega)|\ge r(\omega)$,
$$|f_\omega(z)|\le|f|_n\,\frac{\dt}{1-\dt}.$$
Hence $\sum_\omega|f_\omega|\le\frac{\dt}{1-\dt}N_n|f|_n\le\frac{\dt}{1-\dt}AN_n\rho^n$ outside the discs $D(\omega,r(\omega))$, $\omega\in\Pt_n$, $\pi(\omega)\in U$, which shows that $(f_\omega)$, $\omega\in\Pt$, $\pi(\omega)\in U$, is meandromorphic.

Let $h_n=f_n-\sum_\omega f_\omega$. For each such $\omega$, $f_n-f_\omega$ extends holomorphically to $D(\omega,\l r(\omega))$. The poles with $\pi(\omega)\in\l U\setminus U$ are at distance $\ge\e\ge r(\omega)$ from $(1-\e)U$, so their discs $D(\omega,r(\omega))$ do not meet $(1-\e)\Ut_n$. Hence $h_n$ is holomorphic on $(1-\e)U$. On $\l\Ut(r)_n$ we have $|h_n|\le(1+\frac{\dt}{1-\dt}N_n)|f|_n$. By the maximum principle the same bound holds on $(1-\e)\Ut_n$:
$$\begin{xy}
(0,4)*{\scriptstyle\bullet}; (3,3)*{\omega};
(0,4)*\xycircle(5,5){-};
(0,4)*\xycircle(9,9){-};
(-5,10)*{A_\omega};
(-30,22);(30,22) **\crv{(0,8)};
(40,23)*{(1-\e)\Ut_n};
(-34,6);(34,6) **\crv{(0,-8)};
(38,7)*{\Ut_n};
(-36,-6);(36,-6) **\crv{(0,-18)};
(41,-5)*{\l\Ut_n};
\end{xy}$$
Thus $S(h)=\sum_nh_n$ converges uniformly on $(1-\e)U$ to a holomorphic function, which equals $S(f)-S(P(f))$ on the residual set.
\end{proof}

%%%%%%%%%%%%%%%%%%%%%%%%%%%%%%%%
\subsection{Uniqueness of the polar decomposition}
\label{SS::uniqueness}
We assume that the radius function is an Arnold function. Without loss of generality, we may assume that the projection $\pi:\Pt\to\CM$ is injective. Then the discs $D(\omega,\l r(\omega))$ are pairwise disjoint. The decomposition lives on a single sheet, with one polar part in each disc, as for a meromorphic function. This is what explains the uniqueness: on a single sheet, the polar part $f_\omega$ is determined by the sum function in the annulus around $\pi(\omega)$.

\begin{definition}
\label{D::Borel}
Let $z_0$ be a point of the residual set. We say that a meandromorphic function $f$ has the {\em Borel property} at $z_0$ if its polar decomposition is unique near $z_0$: if there is a holomorphic function germ $\varphi\in\Ot_{z_0}$ such that $S(f)=\varphi$, then $f_\omega=0$ for all the poles $\omega\in\Pt$ close enough to $z_0$.
\end{definition}
For a meromorphic function, whose poles are isolated, this is the uniqueness of the polar parts in the Mittag-Leffler decomposition.

\begin{proposition}
\label{P::uniqueness}
Let $r$, $\l$, $\e$, $\dt$, $U$ and $f$ be as in Proposition~\ref{P::polar}, and let $z_0=0$ be the centre of $U$. Assume moreover that
\begin{enumerate}
\item $z_0\in\Rt_{\l r}(U)$;
\item $r$ is a $\l$-Arnold function.
\end{enumerate}
Then $f$ has the Borel property at $z_0$, i.e.\ its polar decomposition is unique near $z_0$.
\end{proposition}
\begin{proof}
The idea is to sum all the contributions which correspond to the same pole. For a meandromorphic function $f$ and a point $q\in\CM$, we denote by
$$F_q=\sum_{\omega\in\Pt,\ \pi(\omega)=q}f_\omega$$
the sum of the polar parts of $f$ at the poles lying over $q$, and we put
$$r(q):=\sup_{\pi(\omega)=q}r(\omega) .$$
In this way we get a new decomposition, indexed by the points $q$ of the plane and no longer by the points $\omega$ of the stack: the function $F_q$ is holomorphic outside the disc $D(q,r(q))$, and the meandromorphic function $f$ factors through the family $F=(F_q)$, in the sense that
$$S(P(f))=\sum_{\omega}f_\omega=\sum_qF_q .$$
So we may assume that the projection $\pi:\Pt\to\CM$ is injective.

Let $\varphi$ be holomorphic on $D(z_0,\rho)$, $\rho\le1-\e$, with $S(f)=\varphi$ on $\Rt_r(D(z_0,\rho))$. By Proposition~\ref{P::polar}, we have $S(f)=h+\sum_\omega f_\omega$ on $\Rt_r(D(z_0,\rho))$, where $h$ is holomorphic on $(1-\e)U$. Fix $\omega\in\Pt$ with $q=\pi(\omega)\in D(z_0,\rho/2)$. The series
$$\Psi_\omega:=\sum_{\eta\neq\omega}f_\eta$$
is holomorphic in $D(q,\l r(\omega))$. The annulus
$$A:=\{z\in\CM:\ r(\omega)<|z-q|<\l r(\omega)\}$$
is contained in $\Rt_r(D(z_0,\rho))$. On this annulus, the assumption $S(f)=\varphi$ reads
$$f_\omega=\varphi-h-\Psi_\omega .$$
The left-hand side is holomorphic for $|z-q|>r(\omega)$ and tends to zero at infinity. The right-hand side is holomorphic on the disc $D(q,\l r(\omega))$. So both sides glue together to an entire function which tends to zero at infinity. Hence $f_\omega=0$.
\end{proof}

\begin{corollary}
\label{T::nonholomorphic}
Under the assumptions (1) and (2) of Proposition~\ref{P::uniqueness}, the following conditions are equivalent:
\begin{enumerate}
\item[(i)] the points $\pi(\omega)$ such that $f_\omega\ne0$ accumulate at $z_0$;
\item[(ii)] there is no holomorphic function germ $\varphi\in\Ot_{z_0}$ such that $S(f)=\varphi$.
\end{enumerate}
\end{corollary}

 %%%%%%%%%%%%%%%%%%%%%%%%%%%%%%%%%%%%%%%%%%%%%%%%%
\section{A new quasi-analytic class}
\label{S::Gevrey}
\subsection{Quasi-analyticity}
\begin{definition}
A radius function $r:\Pt \to \RM_{>0}$ is called  {\em $(c,\a)$-Diophantine} (or simply $\a$-Diophantine or even Diophantine) if
$$(D):\;\;\; \forall n \geq 1,\;\;  r_n\geq \frac{c}{n^{\a}}.$$
\end{definition}

As by definition meandromorphic sequences are summable, we can form $S(f)$, its {\em meandromorphic sum function},
on the residual set $\Rt_r(V)$.

\begin{theorem}
\label{T::quasianalytic}
Let $r\le1$ be a $(c,\a)$-Diophantine $\l$-radius function with $0<\a<2$, $\e=\l-1$. Let $V'=D(0,1)$ and $V=\l V'$. Assume that $(1+2\e)\,r(\omega)\le\e$ for every pole $\omega$ whose disc $D(\omega,\l r(\omega))$ meets $V'$, and let $z_0\in\Rt_{\l r}(V')$. Assume that near $z_0$ the poles lie on finitely many rays issued from $z_0$ and that every sector of the complement of these rays has width $>\a\pi+2\arcsin(1/\l)$. Then the Taylor map
$$T:S(\Mt^r_{z_0}) \to \CM[[z-z_0]],\ f \mapsto \widehat f=\sum_n \frac{f^{(n)}(z_0)}{n!}(z-z_0)^n $$
is injective.
\end{theorem}
In other words, up to a shrinking of the definition domain, for any meandromorphic germ $f$ at the point $z_0$ with zero asymptotic expansion, there exists a neighbourhood $U \subset \CM$ of $z_0$ such that $S(f)=0$ in $\Rt_r(U)$.

\noindent{\bf Example.} (Poles on one side) Consider the perforation with one pole on each sheet,
\begin{gather*}\Pt=\{\omega_n=(n,1/n):\ n\ge1\},\qquad r(\omega_n)=\frac c{n^\a},\\ \a\ge1,\qquad 0<c\le\min\Bigl(\frac1\l,\frac{\e}{1+2\e}\Bigr),\end{gather*}
and the point $z_0=0$. The function $r\le1$ is a $(c,\a)$-Diophantine $\l$-radius function, and $(1+2\e)\,r(\omega_n)\le\e$ for all $n$. We have $z_0\in\Rt_{\l r}(V')$, since $\l r(\omega_n)\le 1/n=|\pi(\omega_n)-z_0|$. The poles lie on the single ray $\RM_{>0}$, whose complement is a sector of width $2\pi$. So the hypotheses of the theorem are satisfied for
$$1\le\a<2-\frac2\pi\arcsin(1/\l) .$$
\begin{center}
\begin{tikzpicture}[scale=1,>=stealth]
\def\th{19.47}% arcsin(1/lambda), lambda=3
\fill[pattern=north east lines,pattern color=gray!45] (0,0) -- (\th:2.3) arc (\th:{360-\th}:2.3) -- cycle;
\draw[dashed] (0,0) -- (\th:3.6); \draw[dashed] (0,0) -- (-\th:3.6);
\draw (0,0) -- (0:3.9) node[right] {$\RM_{>0}$};
\foreach \s/\k in {2.7/1,1.35/0.8,0.9/0.7,0.675/0.6,0.54/0.5,0.45/0.45,0.386/0.4,0.3375/0.35}
  {\draw (\s,0) circle ({\k*\s/3}); \fill (\s,0) circle (0.02);}
\fill (0,0) circle (0.04);
\node[fill=white,inner sep=1pt] at (-0.3,-0.3) {$z_0$};
\draw[<->] (\th:1.2) arc (\th:{360-\th}:1.2);
\node[fill=white,inner sep=1.5pt] at (150:1.75) {$2\pi-2\arcsin(1/\lambda)$};
\draw (0:3.3) arc (0:\th:3.3);
\draw[<-] (10:3.33) -- (4.0,1.3) node[right] {$\arcsin(1/\lambda)$};
\end{tikzpicture}
\end{center}

\noindent{\bf Example.} (Poles on both sides) Consider the perforation with two poles on each sheet,
\begin{gather*}\Pt=\{\omega_{\pm n}=(n,\pm n^{-\b}):\ n\ge1\},\qquad r(\omega_{\pm n})=\frac c{n^\a},\\ \a\ge\b>0,\qquad 0<c\le\min\Bigl(\frac1\l,\frac{\e}{1+2\e}\Bigr),\end{gather*}
and the point $z_0=0$. As before, the function $r\le1$ is a $(c,\a)$-Diophantine $\l$-radius function, $(1+2\e)\,r(\omega_{\pm n})\le\e$ for all $n$, and $z_0\in\Rt_{\l r}(V')$, since $\l r(\omega_{\pm n})\le n^{-\b}$. The poles lie on the two rays $\RM_{>0}$ and $\RM_{<0}$, whose complement consists of two sectors of width $\pi$. So the hypotheses of the theorem are satisfied for
$$\b\le\a<1-\frac2\pi\arcsin(1/\l) .$$
For $\b\ge1$, for instance for the poles $\pm1/n$, the theorem does not apply. The flat function $1/G_\b(z^{-1})$ of \S\ref{SS::flat}, whose poles are the points $\pm n^{-\b}$ with $\b>1$, shows that quasi-analyticity may indeed fail in this situation.
\begin{center}
\begin{tikzpicture}[scale=1,>=stealth]
\def\th{19.47}% arcsin(1/lambda), lambda=3
\fill[pattern=north east lines,pattern color=gray!45] (0,0) -- (\th:2.3) arc (\th:{180-\th}:2.3) -- cycle;
\fill[pattern=north east lines,pattern color=gray!45] (0,0) -- ({180+\th}:2.3) arc ({180+\th}:{360-\th}:2.3) -- cycle;
\foreach \a in {0,180}
  {\draw[dashed] (0,0) -- ({\a+\th}:3.6); \draw[dashed] (0,0) -- ({\a-\th}:3.6);}
\draw (0,0) -- (0:3.9) node[right] {$\RM_{>0}$};
\draw (0,0) -- (180:3.9) node[left] {$\RM_{<0}$};
\foreach \s/\k in {2.7/1,1.35/0.8,0.9/0.7,0.675/0.6,0.54/0.5,0.45/0.45,0.386/0.4,0.3375/0.35}
  {\draw (\s,0) circle ({\k*\s/3}); \fill (\s,0) circle (0.02);
   \draw (-\s,0) circle ({\k*\s/3}); \fill (-\s,0) circle (0.02);}
\fill (0,0) circle (0.04);
\node[below] at (0,-0.1) {$z_0$};
\draw[<->] (\th:1.2) arc (\th:{180-\th}:1.2);
\node[fill=white,inner sep=1.5pt] at (90:1.65) {$\pi-2\arcsin(1/\lambda)$};
\end{tikzpicture}
\end{center}

We keep the notations of the theorem.

\begin{corollary} Under the assumptions of the previous theorem, if the power series giving the Taylor expansion $T(S(f))$ converges then it is equal to the function $S(f)$ which is therefore a holomorphic function in a neighbourhood of $z_0$.
\end{corollary}
 
In order to prove the above theorem, we will first investigate Gevrey properties of meandromorphic sums $S(f)$
defined by a Diophantine radius function.  Then using the Watson lemma, we easily conclude.\\

We say that a point $\omega \in \Pt$ is {\em singular} (with respect to $r$) for a meandromorphic function $f \in \Mt^r(V)$ if, in the polar decomposition of $f$, the
 function $\sum_{\eta \in \pi^{-1}(\pi(\omega))} f_\eta$ is non-zero.
  \begin{corollary} 
 \label{T::divergence}
Under the assumptions of Theorem~\ref{T::quasianalytic}, assume moreover that the polar parts of $f$ are rational functions of bounded order. Then the projections of the singular points of a meandromorphic function $f \in \Mt^r_{z_0}$
 accumulate at $z_0$ if and only if the Taylor series of $S(f)$ at $z_0$ diverges.
 \end{corollary}
 \begin{proof}
$\implies$ Assume that the Taylor series of $S(f)$ at $z_0$ converges. By Theorem~\ref{T::quasianalytic}, the germ of $S(f)$ is holomorphic, and by the Borel property, i.e.\ the uniqueness of the polar decomposition (Proposition~\ref{P::uniqueness_rays}), the singular points do not accumulate at $z_0$.

$\Longleftarrow$ Assume that the singular points do not accumulate at $z_0$. This means by the polar decomposition (Proposition~\ref{P::polar}) that the germ
of $S(f)$ is holomorphic, hence its Taylor series converges.
\end{proof}

\subsection{Watson's lemma}
By a sector of width $\theta \in (0,2\pi)$ we mean a set obtained from 
\(
\Sigma=\{z\in\CM:0<\arg z<\theta,\ 0<|z|<r\}
\) after translation and rotation. The real half-lines $\RM_{>0}$ and $e^{i\theta}\RM_{>0}$ are called the rays of the sector.  Here $r$ denotes an arbitrary constant.

Near a ray $e^{i\b}\RM_{>0}$, a meandromorphic function decomposes as the sum of a holomorphic function $h_\b$ and  a polar part $p_\b$.
If we decompose the space into sectors which contain the rays, then we have combinatorial data $(h_\b,p_\b)$ which is reminiscent of classical Stokes data. The theorem says that under strong conditions these Stokes data glue together. The proof of the theorem is based on two ingredients: Gevrey
properties of meandromorphic sums under Diophantine conditions and Watson's lemma:

\begin{lemma}\label{L::Watson}
Let $\Sigma$ be a sector of width $>\a\pi$ and \(f:\Sigma \to\CM\) be a holomorphic function ($\a <2$). If there exist
\(A,B>0\) such that
\[
|f(z)|\leq A\exp\!\left(-\frac{B}{|z|^{1/\a}}\right)
\]
then \(f=0\).
\end{lemma}
For a proof see for instance~\cite{Malgrange_resommabilite}. So our purpose is to prove that Diophantine conditions imply Gevrey convergence of asymptotic expansions, a phenomenon observed in the context of KAM theory by Popov~\cite{Popov_KAM}.  
   %%%%%%%%%%%%%%%%%%%%%%%%%%%%%%
\subsection{Gevrey something}

There are at least three classical notions to which the name of Gevrey is attached: {\em Gevrey series}, {\em Gevrey functions} and {\em Gevrey asymptotics}.

{\bf \em Gevrey series:} A formal power series $f:=\sum_{k=0}^{\infty} a_k z^k \in \C[[z]]$
is said to be a {\em Gevrey series of class $\a$}, if the series
\[\sum_{k=0}^{\infty} \frac{a_k}{k!^{\a}} z^k \in \C\{z\},\ \a\geq 0 \]
is a convergent power series.\footnote{The convention $(\a-1)$ instead of $\a$ on the exponent can also be found in the literature and is in fact the original
  definition used by Gevrey.} By Stirling's formula, it is equivalent to asserting the convergence of  the series:
  \[\sum_{k=0}^{\infty} \frac{a_k}{k^{\a k}} z^k \in \C\{z\},\ \a\geq 0 \]
The set of power series of class $\a$ forms a ring; series of Gevrey class
$\a=0$ are just ordinary analytic series. In the first section we have seen
that for a fixed value of $x \in \C$, the series
$$f(x,z)=\sum_{k \geq 0} \frac{x^k}{1+kz} \in \C[[z]],\ |x|<1 $$
is of Gevrey class $1$ in the variable $z$.

{\bf \em Gevrey functions:} For many purposes the notion of Gevrey series is not sufficient, as it only concerns formal power series at a single point. Let $X\subset \C$ be a locally
closed set  and $f: X \lra \CM$ a Whitney $C^{\infty}$-function.
We say that $f$ is a {\em Gevrey function of class $\a$}, if the sequence $(u_n)$ of $C^n$-norms
$$u_n:=\frac{1}{n!}\sup_{z \in X} | f^{(n)}(z)| $$ is bounded by a sequence of the form $(C^n(n!)^{\a})$.
One may speak of a {\em Gevrey germ} of a function, meaning that an estimate of the above type holds in a neighbourhood of a given point.  This defines a sheaf $\Gt^\a_X$ of Gevrey functions on any locally closed subset $X \subset \CM$;
one has $\Gt^\a_X \subset \Gt^\b_X$ if  $\a \le \b$.

{\bf \em  Gevrey asymptotics:} The condition for being a Gevrey function is strong and sometimes hard to verify. There is an important notion intermediate between Gevrey series and Gevrey functions that we will use. By the Taylor formula, if $f$ is a Gevrey function of class $\a$, there exists a constant $C>0$ such that any point $z_0 \in X$ admits a neighbourhood $V \subset \CM$ such that for all $z \in X \cap V$ we have the estimate:
$$(*)\ \left| f(z)-\sum_{k=0}^{n-1} \frac{f^{(k)}(z_0)}{k!} (z-z_0)^k\right| \leq C^n(n!)^\a|z-z_0|^n.$$
This leads to the following definition: we say that a Whitney $C^\infty$ function $f$ defined on a locally closed subset $X$ has {\em $(C,\a)$-Gevrey asymptotics at a point $z_0 \in X$} (or simply $\a$-Gevrey asymptotics if the constant $C$ need not be made explicit), if there exists  a neighbourhood $V$ of $z_0$ such that the above estimate $(*)$ holds. By definition, it is a pointwise notion (unlike the notion of Gevrey function).
%%%%%%%%%%%%%%%%%%%%%%%%%%%%%%%%%%%%%%%%%%%%%%%%%%%%%%%%%%%%%%%%% 
\subsection{ The Gevrey lemma}
We now return to our meandromorphic situation and let $r:\Pt \to \RM_{>0}$ be a $\l=(1+\e)$-radius function and define $r'=(1+\e)r$.

Using the notations of  Lemma \ref{L::Taylor} ($r'=\l r,\ \e=\l-1$) we obtain the:

\begin{corollary}
\label{C::Taylor}
If $r$ is a $(c,\a)$-Diophantine radius function, then the sum function $S(f)$ is Whitney
$C^\infty$ at each point $z_0 \in \Rt_{r'}(V')$, and moreover there is a constant $A$ independent of $m$ such that
$$  |S(f^{(m)})(z_0)| \le \frac{Am!}{c^m\e^m}\sum_{n \geq 0}n^{\a m} \rho^n $$
\end{corollary}
\begin{proof}
We just have to verify the condition of the Taylor Lemma (Lemma~\ref{L::Taylor})  and plug in the definition of meandromorphicity:
\[C_m:=\sum \frac{|f|_n}{r_n^m} \le \sum_{n \geq 0} A \frac{ \rho^n (n^\a)^m}{c^m}=\frac{A}{c^m}\sum_{n \geq 0}n^{\a m} \rho^n < \infty,\]
for some $A>0$ and $\rho<1$.
As $\rho <1$, we get that:
$$|S(f^{(m)})(z_0)| \le \frac{Am!}{\e^mc^m}\sum_{n \geq 0}n^{\a m} \rho^n$$
\end{proof}
From the corollary it follows that for any $z_0 \in \Rt_{\l r}(V')$ we can define a {\em Taylor map}
$$T:S(\Mt^r_{z_0}) \to \CM[[z-z_0]], $$
that assigns to the sum of each meandromorphic germ $f$ at $z_0$ its Taylor series $T(S(f))$.\\

Clearly,  the Taylor lemma can be adapted to many other situations. The important point is that the
radius of the perforation is large with respect to the norm of $f_n$; convergence can be achieved in various ways.
Here we considered algebraic and geometric sequences, but if needed it can be adapted to other situations.

\begin{lemma}
\label{L::Gevrey} Let $V' \subset V$ be as in Lemma~\ref{L::Taylor}. If $r \le 1$ is a $(c,\a)$-Diophantine radius function and $f \in \Mt^r(V)$ is a meandromorphic function whose constants $A,\rho$ are valid on $V$, then the Taylor series of $S(f)$ is of Gevrey class $\a$
at each point $z_0 \in \Rt_{r'}(V')$.
 \end{lemma}
 \begin{proof}
  According to Corollary~\ref{C::Taylor}, there exist constants $A>0,\ \rho<1$ such that
  $$\frac{1}{m!} |S(f^{(m)})(z_0)| \leq \frac{A}{c^m\e^m}\sum_{n \geq 0}n^{\a m} \rho^n.$$
  Choose $\s \in (\rho,1)$, write $\rho=(\rho/\s)\s$ and consider
  the function
$$g: x \mapsto x^{\a m} \s^x .$$
It attains its maximum at
$$ x_{max}=-\frac{\a m}{\log \s}$$
and 
$$g(x_{max})=B^m(\a m)^{\a m},\ B:=\frac{\s^{-\a/\log \s}}{|\log \s|^\a}. $$
Thus 
\begin{align*}
 \frac{A}{c^m\e^m}\sum_{n \geq 0}n^{\a m} \rho^n &\leq A \left(\frac{B\a^{\a}}{c \e}\right)^m m^{\a m} \sum_{n \geq 0} \left(\frac{\rho}{\s} \right)^n \\
 &= A \left(\frac{B\a^{\a}}{c \e}\right)^m m^{\a m}  \frac{\s}{\s-\rho},
\end{align*}
showing the Gevrey $\a$ nature of the function.
\end{proof}
%%%%%%%%%%%%%%%%%%%%%%%%%%

\subsection{Summation to the least term} 
Inside the disc of convergence, a holomorphic function is uniformly approximated by its Taylor series.
In case the Taylor series is divergent with $\a$-Gevrey asymptotics, the degree $m-1$ Taylor polynomial
is still a good approximation if $|z|$ is small with respect to $1/m^\a$. More precisely we have the following statement:
\begin{lemma}
\label{L::flat}
Consider a function $f :X \to \CM$ having $(C,\a)$-Gevrey asymptotic
expansion at the origin with $\a >0$. Then there exist $A,B>0$ 
such that for any $z \in X$ with $|z| <1/C$, we have
 $$| f(z)-T_m(f)(z) |\leq Ae^{-B|z|^{-1/\a}},  $$
 where
 $$m =\left\lfloor \left( C|z|\right)^{-1/\a} \right\rfloor$$ 
 and $T_m(f)$ is the degree $m-1$ Taylor polynomial:
$$T_m(f)=\sum_{k=0}^{m-1} \frac{f^{(k)}(0)}{k!} z^k.$$
\end{lemma}
\begin{proof} By assumption of $(C,\a)$ Gevrey asymptotics for $f$,
 we have for any $m \in \NM$:
 $$|f(z)-T_m(f)(z)|  \leq u_m:=C^m m!^\a |z|^m$$
We set $q:=C |z| <1$, so that $u_m= m!^\a q^m$. Then
 $$u_{m}-u_{m-1}=\left(m^\a q-1\right)(m-1)!^\a q^{m-1}. $$
 The increment is negative for
  $m^\a q-1 \leq 0$. Hence the minimal value for $u_m$ is obtained for
 $$m=\left\lfloor q^{-1/\a} \right\rfloor \sim  q^{-1/\a}.$$
 By Stirling's formula $m! \sim (m/e)^m \sqrt{2\pi m}$ we obtain
 \begin{align*}
  m!^\a &\sim q^{-m} e^{-\a q^{-1/\a}}  (2\pi q^{-1/\a})^{\a/2}\\
   m!^\a q^m&\sim (2\pi)^{\a/2}e^{-\a q^{-1/\a}}q^{-1/2}  =o(e^{-B|z|^{-1/\a}})        
 \end{align*}
 for any $B< \a C^{-1/\a}$.
 This proves the lemma.
\end{proof}

The lemma expresses the paradoxical fact, pointed out by Poincar\'e,
that for a function with Gevrey asymptotics of class $\a >0$, summing the
divergent Taylor series up to its smallest term leads to an {\em exponentially good approximation} to the value of the function.

\subsection{Proof of Theorem~\ref{T::quasianalytic}}
      %%%%%%%%%%%%%%%%%%%%%%%%%%%%%%%%%%%%%%%%%%%%%%%%%%%%%%%%%%%%%%%%%%%%%%%%%%%%%% 
Let $f\in\Mt^r_{z_0}$ be a flat function.

We choose $\mu\in(1,\l)$ such that every sector of the complement of the rays has width $>\a\pi+2\arcsin(\mu/\l)$. The function $r$ is also a $\mu$-radius function and $z_0\in\Rt_{\mu r}(V')$. Let $s>0$ be such that $D(z_0,s)\subset V'$ and $s\le(\l-\mu)(\mu-1)/(2\mu-1)$. As $z_0\in\Rt_{\l r}(V')$, a pole $\omega$ whose disc $D(\omega,\mu r(\omega))$ meets $D(z_0,s)$ satisfies $r(\omega)<s/(\l-\mu)$, hence $(2\mu-1)\,r(\omega)\le\mu-1$. So we may apply the previous results on $D(z_0,s)$ with $\mu$ in place of $\l$, that is with $r'=\mu r$ and $\e=\mu-1$.

By the Taylor lemma, Lemma~\ref{L::Gevrey} and Lemma~\ref{L::flat}, we get $|S(f)(z)|\le Ae^{-B|z-z_0|^{-1/\a}}$ for $z\in\Rt_{\mu r}(D(z_0,s))$.

As $|\pi(\omega)-z_0|\ge\l r(\omega)$, a disc $D(\omega,\mu r(\omega))$ is seen from $z_0$ inside the cone of half-angle $\arcsin(\mu/\l)$ around its ray:
\begin{center}
\begin{tikzpicture}[scale=0.95,>=stealth]
\def\th{23}% arcsin(mu/lambda)
% reduced sectors
\fill[pattern=north east lines,pattern color=gray!45] (0,0) -- (\th:3.4) arc (\th:{150-\th}:3.4) -- cycle;
\fill[pattern=north west lines,pattern color=gray!45] (0,0) -- ({150+\th}:3.4) arc ({150+\th}:{360-\th}:3.4) -- cycle;
% cones around the rays
\foreach \a in {0,150}
  {\draw[dashed] (0,0) -- ({\a-\th}:3.7); \draw[dashed] (0,0) -- ({\a+\th}:3.7);}
% rays
\draw (0,0) -- (0:4.0) node[right] {$L_1$};
\draw (0,0) -- (150:4.0) node[above left] {$L_2$};
% discs D(omega, mu r(omega))
\foreach \a/\s/\k in {0/2.6/1,0/1.08/0.8,0/0.45/1,0/0.18/0.9,150/2.4/0.8,150/1.0/1,150/0.42/0.7,150/0.17/1}
  {\draw (\a:\s) circle ({\k*\s*sin(\th)}); \fill (\a:\s) circle (0.025);}
\fill (0,0) circle (0.04);
\node[below] at (0.05,-0.12) {$z_0$};
\node[fill=white,inner sep=1.5pt] at (75:2.0) {$\Sigma$};
% width of the reduced sector
\draw[<->] (\th:1.15) arc (\th:{150-\th}:1.15);
\node[fill=white,inner sep=1pt] at (75:1.42) {$>\alpha\pi$};
% half-angle of the cone
\draw (0:3.75) arc (0:\th:3.75);
\draw[<-] (12:3.78) -- (4.3,1.6) node[right] {$\arcsin(\mu/\lambda)$};
% label of a disc
\draw[<-] (2.6,-0.55) -- (3.5,-1.7) node[below] {$D(\omega,\mu r(\omega))$};
\end{tikzpicture}
\end{center}
Let $\theta_1<\theta_2$ be the arguments of two consecutive rays. We may choose the sector
$$\Sigma=\{z\in\CM:\ 0<|z-z_0|<s,\ \theta_1+\arcsin(\mu/\l)<\arg(z-z_0)<\theta_2-\arcsin(\mu/\l)\},$$
which has width $>\a\pi$ and avoids these discs. On it all the $f_n$ are holomorphic and $\sum f_n$ converges uniformly, so $S(f)$ is holomorphic there, and Watson's lemma gives $S(f)=0$ on it. By Proposition~\ref{P::polar}, $S(f)=h+G$ on the residual set, where $h$ is holomorphic near $z_0$ and $G$ is the sum of the polar parts. Since $G=-h$ on the sectors $\Sigma$, analytic continuation gives $G=-h$ in the interior of the residual set, so $S(f)=0$ there. Near $z_0$, the residual set is the closure of its interior; by continuity of $S(f)$, we get $S(f)=0$ on the residual set near $z_0$.
     \appendix
 \section{The Poincar\'e Meander}
 \subsection{As a formal power series}
 The Poincar\'e meander is defined by the expansion:
$$f(x,z)=\sum_{n \geq 1} \frac{x^n}{1+nz}. $$
It was intuited by Poincar\'e that the divergent series appearing in perturbation theory should be of this sort~\cite[VIII\ \S 119]{Poincare_Methodes_2}.
Expanding the denominators we get the expressions:
\begin{align*}
f(x,z)&= \sum_{j \geq 0,\ n \geq 1} (-1)^jx^n n^j z^j \\
     &=\sum_{j \geq 0} (-1)^j \theta^j\left(\frac{1}{1-x}\right) z^j
\end{align*}
where $\theta=x\d_x$. We have
$$ \theta^j\left(\frac{1}{1-x}\right)=\frac{xE_j(x)}{(1-x)^{j+1}},\ j \geq 1$$
where $E_j$ are known as the {\em Eulerian polynomials}:
\[E_1(x)=1,\;\;\;E_2(x)=1+x,\;\;\;E_3(x)=1+4x+x^2,\;\;\;E_4(x)=1+11x+11x^2+x^3, \ldots\]

For fixed $x$, the coefficients do not grow faster than $C^n n!$. Indeed denote by $| \cdot |_t$ the supremum norm in the disc
$$D_t=\{ x \in \CM:|x| < t \} \text{ with } t<1.$$
The Cauchy inequality and the maximum principle imply that:
$$|\theta(f)|_s \leq \frac{t}{t-s}|f|_t \leq \frac{1}{t-s}|f|_t\quad \text{ for } t \leq 1$$
So if we divide the segment $(t-s)$ into $j$ equal parts $\e=(t-s)/j$, we get that
$$|\theta^j(f)|_s \leq 
\e^{-1} | \theta^{j-1}(f) |_{s+\e}
  \leq 
\e^{-2}|\theta^{j-2}(f)|_{s+2\e}
\leq \dots \leq \e^{-j}|f|_t=\frac{j^j}{(t-s)^j} |f|_t$$
and $j^j \sim j!e^j(2\pi j)^{-1/2}$ (Stirling formula).

Observe that the coefficients of the series do not grow slower than a factorial as for $x \in \RM_{>0}$, we have:
$$\sum_{n \geq 0}  x^n n^j > x^j j^j \sim  j!x^je^j(2\pi j)^{-1/2}$$
So the Poincar\'e meander is a divergent Gevrey class $1$ series.
%%%%%%%%%%%%%%%%%%%%%%
\subsection{Meandromorphic properties}
  We consider the set of poles 
$$\Pt=\{ (n,z): 1+nz=0 \} \subset \X:=\NM \times \CM $$
We have a holomorphic function
$$f:\X \setminus \Pt \to \C,\ (n,z) \mapsto f_n(z):=\frac{x^n}{1+nz}  $$
where $x$ is a fixed parameter inside the unit disc.
 Choose a Diophantine radius function
$$r:\Pt \to \RM_{>0},\ (n,z) \mapsto \frac{C}{n^\a},\ 1<\a<2,\ \l C<1.$$
As $\l Cn^{-\a}<1/n$ for all $n$, the origin lies in the residual set $\Rt_{\l r}(\CM)$. In each sheet of $\X$ there is exactly one pole and therefore one polar disc: 
$$D(p,r(p))=\left\{ (n,z) \in \X: |z+\frac{1}{n}| < \frac{C}{n^\a} \right\},\ p=(n,-1/n). $$
Outside the polar disc we have the estimate
$$\left|\frac{x^n}{1+nz} \right| \leq C^{-1}n^{\a-1}|x|^n,\ z \notin D(\omega,r(\omega)) $$
and therefore the function $f$ is meandromorphic for any fixed $x$ inside the unit disc. The sum function is,
by the Taylor lemma (Lemma~\ref{L::Taylor}), Whitney $C^\infty$ on $\Rt_{\l r}(V')$ and, by the Gevrey lemma (Lemma~\ref{L::Gevrey}), its Taylor series are of Gevrey class $\a$.

 So if we fix $x \in \CM$ with $|x|<1$ then
the Poincar\'e meander is Whitney $C^\infty$ holomorphic (hence monogenic which corresponds to the $C^1$-case) inside the set
$$X=\{ z \in \CM: \exists C>0,\exists \a \in(1,2),\ \forall n \in \NM,  |z+\frac{1}{n}| > \frac{C}{n^\a} \}$$
and it is the only meandromorphic sum having $ \sum_{j \geq 0,\ n \geq 1} (-1)^jx^n n^j z^j$ as Taylor expansion at the origin (Theorem~\ref{T::quasianalytic}: the poles lie on a single ray).

%%%%%%%%%%%%%%%%%%%%%%%
\section{The Euler  $q$-logarithm}
\subsection{Euler's paper}
In the paper E190, Euler studies the series
\[ \frac{1}{1-a}(1-x)+\frac{1}{1-a^2}(1-x)\left(1-\frac{x}{a}\right)+\frac{1}{1-a^3}(1-x)\left(1-\frac{x}{a}\right)\left(1-\frac{x}{a^2}\right)+\ldots\]
that interpolates the logarithm to base $a$, as it attains the value $n$ for
$x=a^n$. For $x=0$, he was led to consider the series
\[ \frac{1}{a-1}+\frac{1}{a^2-1}+\frac{1}{a^3-1}+\frac{1}{a^4-1}+\ldots\]
which in $\S 29$ of that paper is generalised to
\[ \frac{1}{a-x}+\frac{1}{a^2-x}+\frac{1}{a^3-x}+\frac{1}{a^4-x}+\ldots,\]
which by expansion in geometric series he brings into the form
\[ s=\frac{1}{a-1}+\frac{x}{a^2-1}+\frac{x^2}{a^3-1}+\frac{x^3}{a^4-1}+\ldots=\sum_{n=1}^{\infty}\frac{x^{n-1}}{a^n-1} ,\]  
which we recognise  to be the series of the {\em quantum logarithm} up to a factor:
\[ \log_q(1-z):=z s=\sum_n \frac{z^n}{q^n-1}.\]

 The analytic properties of these series were studied by Marmi and Sauzin in~\cite{Marmi_Sauzin}. We will change the naming of the
 variables and study 
 $$xL(x,z)=\sum_n \frac{x^n}{z^n-1}$$ for fixed $x$ as a function of the variable $z$.
 %%%%%%%%%%%%%%%%%%%%%
\subsection{A radius function}
\label{SS::radius}
We set 
$\Pt=\{(n,\omega) \in  \X: \omega^n=1 \}$
and  consider a Diophantine radius function
$$r: \Pt \to \RM_{>0},\ (n,\omega) \mapsto \frac{C}{n^\a},$$
for $\a\ge2$ and $C>0$ small enough so that $\l Cn^{-\a}<\sin(\pi/n)$ for all $n\ge2$. A point $z_0=e^{2i\pi\theta_0}$ belongs to $\Rt_{\l r}(\CM)$ if and only if $\|n\theta_0\|\ge\frac{\l C}{2\pi}n^{1-\a}$ for all $n$; in particular $C$ depends on $z_0$.  We get a holomorphic function
 $$f:\X \setminus \Pt \longrightarrow \C,\ (n,z) \mapsto \frac{x^n}{z^n-1} $$
 Using the partial fraction decomposition
$$\frac1{z^n-1}=\frac1n\sum_{\omega^n=1}\frac{\omega}{z-\omega},$$
we get outside the discs
$$\left| \frac{x^n}{z^n-1} \right| \leq |x|^n\max_{\omega^n=1}\frac1{|z-\omega|} \leq
C^{-1}n^\a|x|^n,$$
showing that the function $f$ is meandromorphic and hence its sum is  $C^\infty$ on the residual set $\Rt_{\l r}(\CM)$ and even of Gevrey class $\a$.

For instance, the segment $L=\{ t e^{2\pi i \sqrt{2}} \in \CM: t \in [0,2) \} $ is contained inside  this residual set, for $C$ small enough. This is the continuation along rays discovered by Borel~\cite{Borel_monogene,Borel1917} that we discussed in \S1.2.
 %%%%%%%%%%%%
 \subsection{Polar decomposition}
We have:
$$\frac{1}{z^n-1}=\frac{1}{n}\sum_{k=1}^n \frac{\omega_n^k}{z-\omega_n^k},\
\omega_n=e^{2i\pi/n} $$
and therefore we have the polar decomposition:
$$g: \X[\Pt] \to \CM,\ (\omega,z) \mapsto \frac{x^n}{n}\frac{\omega}{z-\omega},\
\omega=e^{2ik\pi/n}  $$
which is meandromorphic with the same sum as our initial function $f$.

\begin{figure}[htb!]
\centering
\includegraphics[width=0.6\linewidth]{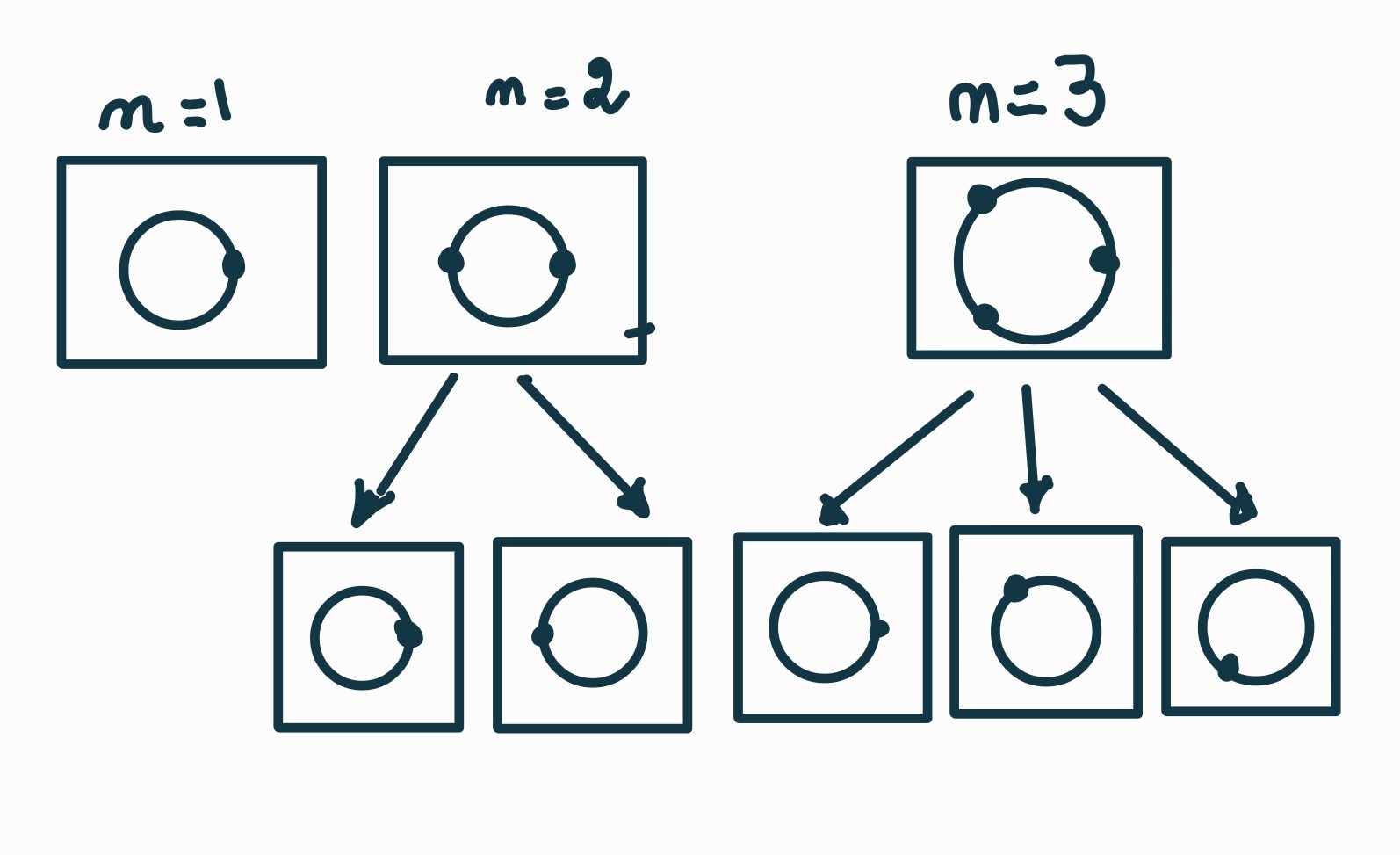}\\
{\tiny First sheets of the polar decomposition for $\log_q$}
\end{figure}

%%%%%%%%%%%%
\subsection{Induced radius functions and Arnold's covering theorem}
The discs of the radius function defined in \S\ref{SS::radius} do not satisfy the Arnold property. It is however possible to find a similar
radius function with the Arnold property. Let us first explain what similar means.
\begin{definition} We say that a $\l$-radius function $r:\Pt \to \RM_{>0} $ is induced by a $\mu$-radius function $r':\Pt' \to \RM_{>0} $
if there is a surjection
$$f:\Pt \to \Pt'$$ such that for any $\omega' \in \Pt'$, one has the inclusions:
$$\bigcup_{f(\omega)=\omega'} D(\omega,r) \subset D(\omega',r') \subset  D(\omega',\mu r') \subset \bigcup_{f(\omega)=\omega'} D(\omega,\l r) .$$
\end{definition}
The first inclusion shows that a function which is meandromorphic for $r$ is still meandromorphic for $r'$, and the last one that the residual set of $\l r$ is contained in that of $\mu r'$.

Arnold chose a real model in $[0,1]$ rather than roots of unity on the circle.
The roots of unity are the images of rational points in the segment $[0,1]$ via the exponential map $x \mapsto e^{2i\pi x}$. Asymptotically the image of
a disc centred at $\omega \in \CM$ via the exponential map is a disc:
$$e^{\omega+re^{i\theta}}-e^\omega \sim re^\omega e^{i\theta}  $$
Therefore, in order to prove the existence of a Diophantine Arnold radius function, we consider poles defined by
$$\Pt_{n-1}=\left\{ \frac{k}{n}: k \in \llbracket 0 ,n\rrbracket \right \}$$
and then take a Diophantine $\l$-radius function of the form:
$$r:\Pt \to \RM_{>0},\ \frac{k}{n} \mapsto \frac{ c}{n^\a}\ \text{with }\a \geq 3,\ c>0. $$

\begin{theorem}[{\cite[\S 7]{Arnold_SD1}}]
\label{T::Arnold} For any $\a \geq 3$ there exist constants $C_\a>0$ and $\mu>1$ such that the Diophantine
radius function $r:\Pt \to \RM_{>0},\ \frac{k}{n} \mapsto \frac{c}{n^\a}$ is induced by a $\mu$-Arnold radius function for $c \leq C_\a$.
\end{theorem}
The radii of an induced Arnold function are larger than those of $r$, so the function $f$ is still meandromorphic for it, and Proposition~\ref{P::uniqueness} applies to the quantum logarithm, the discs of the Arnold function playing the role of the poles.

 %%%% 

\section{Polar decomposition of rational sums}
\label{SS::rays}
For sums of rational functions with bounded order, there is an elementary proof of the Borel property (Definition~\ref{D::Borel}), i.e.\ of the uniqueness of the polar decomposition,
 without imposing the Arnold condition, under a simple condition given below. Note however that in perturbation theory, the poles have, as a general rule, infinite order~(see e.g. \cite{HNF,Versal_fields}).
 
As in Proposition~\ref{P::polar}, let $r\le\e$ be a $\l$-radius function on a discrete set of poles $\Pt\subset\X$, with $\e=\l-1<1$ and $0<\dt<\e$, let $U$ be the open unit disc, $z_0=0$ its centre, and let $f \in \Mt^{\dt r}(\overline{\l U})$ be a meandromorphic function, with polar parts $f_\omega$, $\omega\in\Pt$.

  Without loss of generality, we may assume that the projection $\pi:\Pt\to\CM$ becomes injective by summing the polar parts which have the same projection. In particular, we identify $\Pt$ with its projection.

\begin{definition}
Let $D\subset U$ be an open disc and $D'=D\setminus\overline{\Pt \cap D}$. A pole $\omega\in \Pt\cap D$ is called {\em accessible in $D$} if there exist a constant $c>0$ and a path $\g:[0,1]\to D$ with $\g(1)=\omega$ and $\g(t)\in D'$ for $t<1$, such that
\begin{enumerate}
\item $\g(0)$ lies in the interior of the residual set $\Rt_r(D)$;
\item $|\g(t)-\omega'|\ge c\,|\g(t)-\omega|$ for all $t<1$ and all $\omega'\in \Pt$, $\omega'\neq \omega$.
\end{enumerate}
\end{definition}
For instance, assume that $z_0\in\Rt_{\l r}(U)$, that the poles lie on finitely many rays issued from $z_0$ and that every sector of the complement of these rays has width $>2\arcsin(1/\l)$. Then every pole of $D(z_0,\rho)$ is accessible in $D(z_0,\rho)$ for $\rho\le1-\e$.
 Seen from $z_0$, each polar disc lies in the cone of half-angle $\arcsin(1/\l)$ around its ray, so each sector contains an open subset of the residual set, and one approaches a pole perpendicularly to its ray:
\begin{center}
\begin{tikzpicture}[scale=0.9,>=stealth]
\def\th{19.47}% arcsin(1/lambda), lambda=3
\foreach \a in {0,95,215}
  {\fill[gray!35] (0,0) -- ({\a-\th}:3.9) arc ({\a-\th}:{\a+\th}:3.9) -- cycle;
   \draw[dashed] (0,0) -- ({\a-\th}:3.9); \draw[dashed] (0,0) -- ({\a+\th}:3.9);}
\fill[pattern=north east lines,pattern color=gray!60] (0,0) -- (\th:1.5) arc (\th:{95-\th}:1.5) -- cycle;
\draw (0,0) -- (0:3.9) node[right] {$L_1$};
\draw (0,0) -- (95:3.9) node[above] {$L_2$};
\draw (0,0) -- (215:3.9) node[below left] {$L_3$};
\fill (0,0) circle (0.035) node[below right] {$z_0$};
\foreach \a/\s/\k in {0/2.7/1,0/1.75/0.9,0/1.1/1,0/0.7/0.85,0/0.44/1,0/0.27/0.9,0/0.17/1,95/2.9/0.9,95/1.9/1,95/1.2/0.85,95/0.76/1,95/0.48/0.9,95/0.3/1,95/0.19/0.9,215/2.5/1,215/1.6/0.85,215/1.0/1,215/0.63/0.9,215/0.4/1,215/0.25/0.85,215/0.16/1}
  {\draw (\a:\s) circle ({\k*\s/3}); \fill (\a:\s) circle (0.025);}
\node[below] at (2.7,-0.02) {$\omega$};
\fill (47:1.05) circle (0.035);
\draw[thin] (47:1.05) -- (1.25,2.0) node[above] {$\g(0)$};
\draw[thick,->] plot[smooth,tension=0.75] coordinates {(47:1.05) (1.25,1.45) (1.6,0.85) (2.05,1.4) (2.35,0.8) (2.7,0.75) (2.7,0.04)};
\node[above] at (2.05,1.42) {$\g$};
\draw (0:2.0) arc (0:-\th:2.0);
\draw[<-] (-10:2.03) -- (3.3,-2.0) node[right] {$\arcsin(1/\lambda)$};
\end{tikzpicture}
\end{center}

We assume that the polar parts are rational functions whose order is bounded uniformly:  for any $\omega\in \Pt$, there is an integer $K=K_\omega\ge1$,
$$f_\omega(z)=\sum_{k=1}^K\frac{c_{k,\omega}}{(z-\pi(\omega))^k} .$$
 We put:
$$f_\omega=\sum_{k=1}^K\frac{c_{k,\omega}}{(z-\omega)^k},\qquad P=\sum_{\omega \in \Pt \cap U}f_\omega .$$
By the Cauchy estimates of \S\ref{SS::polar}, we have $\sum_{\omega,k}|c_{k,\omega}|<\infty$. Hence $P$ is holomorphic on the complement $\Omega:=\CM\setminus\overline{\Pt}$ of the closure of the set of poles.

Uniqueness of the decomposition follows from the:
\begin{proposition}
\label{P::uniqueness_rays}
Let $\rho\in(0,1-\e]$ be such that every pole of $D(z_0,\rho)$ is accessible in $D(z_0,\rho)$. If there exists a holomorphic function $\varphi$ on $D(z_0,\rho)$ such that $S(f)=\varphi$ on $\Rt_r(D(z_0,\rho))$, then $f_\omega=0$ for every $\omega\in \Pt\cap D(z_0,\rho)$.
\end{proposition}
\begin{proof}
By Proposition~\ref{P::polar}, we have $S(f)=h+P$ on the residual set, where $h$ is holomorphic on $(1-\e)U$, so $P=\varphi-h$ on $\Rt_r(D(z_0,\rho))$. By analytic continuation, this identity holds along the paths of the definition.
Assume $P $ has a pole at $\omega$ and let $(z_m)$ be a sequence of points of the path $\g$ converging to $\omega$. Assume, by descending induction on $j\le K$, that $c_{k,\omega'}=0$ for $k>j$ and all the poles $\omega'$ of $D(z_0,\rho)$. By (2), $(z_m-\omega)^jP(z_m)$ tends to $c_{j,\omega}$, whereas $(z_m-\omega)^j(\varphi-h)(z_m)$ tends to $0$. Hence $c_{j,\omega}=0$.
\end{proof}

\subsection*{Acknowledgements}
We were surprised and delighted to have R. P\'erez-Marco as a non-anonymous referee. We wish to take this opportunity to express our gratitude to him; his work has long been, and will remain, a source of inspiration for us. We have followed most of his suggestions to improve and simplify the paper, though not all of them. We also thank Claude (Anthropic) and ChatGPT (OpenAI) for technical assistance and for checking computations and parts of the proofs.

    \bibliographystyle{amsplain}
\bibliography{master}
\end{document}

%% file: Fig/Wolff_tikz.tex
% Apollonian packing of the unit disc (Wolff's example), generated by D(+-1/2,1/2), D(+-2i/3,1/3)
\begin{tikzpicture}[scale=2.3]
\fill[gray!20] (-1.35,-1.75) rectangle (1.35,1.35);
\fill[white] (0,0) circle (1);
\draw[semithick] (0.50000,0.00000) circle (0.50000);
\draw[semithick] (-0.50000,0.00000) circle (0.50000);
\draw[semithick] (0.00000,0.66667) circle (0.33333);
\draw[semithick] (-0.00000,-0.66667) circle (0.33333);
\draw[semithick] (0.50000,0.66667) circle (0.16667);
\draw[semithick] (0.50000,-0.66667) circle (0.16667);
\draw[semithick] (-0.50000,0.66667) circle (0.16667);
\draw[semithick] (-0.50000,-0.66667) circle (0.16667);
\draw[thin] (0.72727,0.54545) circle (0.09091);
\draw[thin] (0.72727,-0.54545) circle (0.09091);
\draw[thin] (-0.72727,0.54545) circle (0.09091);
\draw[thin] (-0.72727,-0.54545) circle (0.09091);
\draw[thin] (0.35714,0.85714) circle (0.07143);
\draw[thin] (0.35714,-0.85714) circle (0.07143);
\draw[thin] (-0.35714,0.85714) circle (0.07143);
\draw[thin] (-0.35714,-0.85714) circle (0.07143);
\draw[thin] (0.00000,0.26667) circle (0.06667);
\draw[thin] (0.00000,-0.26667) circle (0.06667);
\draw[thin] (0.83333,0.44444) circle (0.05556);
\draw[thin] (0.83333,-0.44444) circle (0.05556);
\draw[thin] (-0.83333,0.44444) circle (0.05556);
\draw[thin] (-0.83333,-0.44444) circle (0.05556);
\draw[thin] (0.34783,0.52174) circle (0.04348);
\draw[thin] (0.34783,-0.52174) circle (0.04348);
\draw[thin] (-0.34783,0.52174) circle (0.04348);
\draw[thin] (-0.34783,-0.52174) circle (0.04348);
\draw[thin] (0.26923,0.92308) circle (0.03846);
\draw[thin] (0.26923,-0.92308) circle (0.03846);
\draw[thin] (-0.26923,0.92308) circle (0.03846);
\draw[thin] (-0.26923,-0.92308) circle (0.03846);
\draw[thin] (0.88889,0.37037) circle (0.03704);
\draw[thin] (0.88889,-0.37037) circle (0.03704);
\draw[thin] (-0.88889,0.37037) circle (0.03704);
\draw[thin] (-0.88889,-0.37037) circle (0.03704);
\draw[thin] (0.70000,0.66667) circle (0.03333);
\draw[thin] (0.70000,-0.66667) circle (0.03333);
\draw[thin] (-0.70000,0.66667) circle (0.03333);
\draw[thin] (-0.70000,-0.66667) circle (0.03333);
\draw[thin] (0.45714,0.85714) circle (0.02857);
\draw[thin] (0.45714,-0.85714) circle (0.02857);
\draw[thin] (-0.45714,0.85714) circle (0.02857);
\draw[thin] (-0.45714,-0.85714) circle (0.02857);
\draw[thin] (0.00000,0.17143) circle (0.02857);
\draw[thin] (0.00000,-0.17143) circle (0.02857);
\draw[thin] (0.92105,0.31579) circle (0.02632);
\draw[thin] (0.92105,-0.31579) circle (0.02632);
\draw[thin] (-0.92105,0.31579) circle (0.02632);
\draw[thin] (-0.92105,-0.31579) circle (0.02632);
\draw[thin] (0.07895,0.31579) circle (0.02632);
\draw[thin] (-0.07895,0.31579) circle (0.02632);
\draw[thin] (0.07895,-0.31579) circle (0.02632);
\draw[thin] (-0.07895,-0.31579) circle (0.02632);
\draw[thin] (0.61538,0.51282) circle (0.02564);
\draw[thin] (0.61538,-0.51282) circle (0.02564);
\draw[thin] (-0.61538,0.51282) circle (0.02564);
\draw[thin] (-0.61538,-0.51282) circle (0.02564);
\draw[very thin] (0.21429,0.95238) circle (0.02381);
\draw[very thin] (0.21429,-0.95238) circle (0.02381);
\draw[very thin] (-0.21429,0.95238) circle (0.02381);
\draw[very thin] (-0.21429,-0.95238) circle (0.02381);
\draw[very thin] (0.34043,0.76596) circle (0.02128);
\draw[very thin] (0.34043,-0.76596) circle (0.02128);
\draw[very thin] (-0.34043,0.76596) circle (0.02128);
\draw[very thin] (-0.34043,-0.76596) circle (0.02128);
\draw[very thin] (0.30000,0.48000) circle (0.02000);
\draw[very thin] (0.30000,-0.48000) circle (0.02000);
\draw[very thin] (-0.30000,0.48000) circle (0.02000);
\draw[very thin] (-0.30000,-0.48000) circle (0.02000);
\draw[very thin] (0.94118,0.27451) circle (0.01961);
\draw[very thin] (0.94118,-0.27451) circle (0.01961);
\draw[very thin] (-0.94118,0.27451) circle (0.01961);
\draw[very thin] (-0.94118,-0.27451) circle (0.01961);
\draw[very thin] (0.83333,0.51852) circle (0.01852);
\draw[very thin] (0.83333,-0.51852) circle (0.01852);
\draw[very thin] (-0.83333,0.51852) circle (0.01852);
\draw[very thin] (-0.83333,-0.51852) circle (0.01852);
\draw[very thin] (0.67797,0.71186) circle (0.01695);
\draw[very thin] (0.40678,0.50847) circle (0.01695);
\draw[very thin] (0.67797,-0.71186) circle (0.01695);
\draw[very thin] (0.40678,-0.50847) circle (0.01695);
\draw[very thin] (-0.67797,0.71186) circle (0.01695);
\draw[very thin] (-0.40678,0.50847) circle (0.01695);
\draw[very thin] (-0.67797,-0.71186) circle (0.01695);
\draw[very thin] (-0.40678,-0.50847) circle (0.01695);
\draw[very thin] (0.17742,0.96774) circle (0.01613);
\draw[very thin] (0.33871,0.58065) circle (0.01613);
\draw[very thin] (0.17742,-0.96774) circle (0.01613);
\draw[very thin] (0.33871,-0.58065) circle (0.01613);
\draw[very thin] (-0.17742,0.96774) circle (0.01613);
\draw[very thin] (-0.33871,0.58065) circle (0.01613);
\draw[very thin] (-0.17742,-0.96774) circle (0.01613);
\draw[very thin] (-0.33871,-0.58065) circle (0.01613);
\draw[very thin] (0.76190,0.44444) circle (0.01587);
\draw[very thin] (0.76190,-0.44444) circle (0.01587);
\draw[very thin] (-0.76190,0.44444) circle (0.01587);
\draw[very thin] (-0.76190,-0.44444) circle (0.01587);
\draw[very thin] (0.00000,0.12698) circle (0.01587);
\draw[very thin] (0.00000,-0.12698) circle (0.01587);
\draw[very thin] (0.95455,0.24242) circle (0.01515);
\draw[very thin] (0.50000,0.84848) circle (0.01515);
\draw[very thin] (0.95455,-0.24242) circle (0.01515);
\draw[very thin] (0.50000,-0.84848) circle (0.01515);
\draw[very thin] (-0.95455,0.24242) circle (0.01515);
\draw[very thin] (-0.50000,0.84848) circle (0.01515);
\draw[very thin] (-0.95455,-0.24242) circle (0.01515);
\draw[very thin] (-0.50000,-0.84848) circle (0.01515);
\draw[very thin] (0.11268,0.33803) circle (0.01408);
\draw[very thin] (-0.11268,0.33803) circle (0.01408);
\draw[very thin] (0.11268,-0.33803) circle (0.01408);
\draw[very thin] (-0.11268,-0.33803) circle (0.01408);
\draw[very thin] (0.74324,0.64865) circle (0.01351);
\draw[very thin] (0.74324,-0.64865) circle (0.01351);
\draw[very thin] (-0.74324,0.64865) circle (0.01351);
\draw[very thin] (-0.74324,-0.64865) circle (0.01351);
\draw[very thin] (0.32000,0.93333) circle (0.01333);
\draw[very thin] (0.32000,-0.93333) circle (0.01333);
\draw[very thin] (-0.32000,0.93333) circle (0.01333);
\draw[very thin] (-0.32000,-0.93333) circle (0.01333);
\draw[very thin] (0.96386,0.21687) circle (0.01205);
\draw[very thin] (0.57831,0.50602) circle (0.01205);
\draw[very thin] (0.96386,-0.21687) circle (0.01205);
\draw[very thin] (0.57831,-0.50602) circle (0.01205);
\draw[very thin] (-0.96386,0.21687) circle (0.01205);
\draw[very thin] (-0.57831,0.50602) circle (0.01205);
\draw[very thin] (-0.96386,-0.21687) circle (0.01205);
\draw[very thin] (-0.57831,-0.50602) circle (0.01205);
\draw[very thin] (0.89535,0.41860) circle (0.01163);
\draw[very thin] (0.15116,0.97674) circle (0.01163);
\draw[very thin] (0.89535,-0.41860) circle (0.01163);
\draw[very thin] (0.15116,-0.97674) circle (0.01163);
\draw[very thin] (-0.89535,0.41860) circle (0.01163);
\draw[very thin] (-0.15116,0.97674) circle (0.01163);
\draw[very thin] (-0.89535,-0.41860) circle (0.01163);
\draw[very thin] (-0.15116,-0.97674) circle (0.01163);
\draw[very thin] (0.27586,0.87356) circle (0.01149);
\draw[very thin] (0.27586,0.45977) circle (0.01149);
\draw[very thin] (0.27586,-0.87356) circle (0.01149);
\draw[very thin] (0.27586,-0.45977) circle (0.01149);
\draw[very thin] (-0.27586,0.87356) circle (0.01149);
\draw[very thin] (-0.27586,0.45977) circle (0.01149);
\draw[very thin] (-0.27586,-0.87356) circle (0.01149);
\draw[very thin] (-0.27586,-0.45977) circle (0.01149);
\draw[very thin] (0.43333,0.88889) circle (0.01111);
\draw[very thin] (0.43333,-0.88889) circle (0.01111);
\draw[very thin] (-0.43333,0.88889) circle (0.01111);
\draw[very thin] (-0.43333,-0.88889) circle (0.01111);
\draw[very thin] (0.84211,0.37895) circle (0.01053);
\draw[very thin] (0.67368,0.63158) circle (0.01053);
\draw[very thin] (0.84211,-0.37895) circle (0.01053);
\draw[very thin] (0.67368,-0.63158) circle (0.01053);
\draw[very thin] (-0.84211,0.37895) circle (0.01053);
\draw[very thin] (-0.67368,0.63158) circle (0.01053);
\draw[very thin] (-0.84211,-0.37895) circle (0.01053);
\draw[very thin] (-0.67368,-0.63158) circle (0.01053);
\draw[very thin] (0.66327,0.73469) circle (0.01020);
\draw[very thin] (0.64286,0.48980) circle (0.01020);
\draw[very thin] (0.33673,0.73469) circle (0.01020);
\draw[very thin] (0.66327,-0.73469) circle (0.01020);
\draw[very thin] (0.64286,-0.48980) circle (0.01020);
\draw[very thin] (0.33673,-0.73469) circle (0.01020);
\draw[very thin] (-0.66327,0.73469) circle (0.01020);
\draw[very thin] (-0.64286,0.48980) circle (0.01020);
\draw[very thin] (-0.33673,0.73469) circle (0.01020);
\draw[very thin] (-0.66327,-0.73469) circle (0.01020);
\draw[very thin] (-0.64286,-0.48980) circle (0.01020);
\draw[very thin] (-0.33673,-0.73469) circle (0.01020);
\draw[very thin] (0.00000,0.10101) circle (0.01010);
\draw[very thin] (0.00000,-0.10101) circle (0.01010);
\draw[very thin] (0.97059,0.19608) circle (0.00980);
\draw[very thin] (0.97059,-0.19608) circle (0.00980);
\draw[very thin] (-0.97059,0.19608) circle (0.00980);
\draw[very thin] (-0.97059,-0.19608) circle (0.00980);
\draw[very thin] (0.02941,0.19608) circle (0.00980);
\draw[very thin] (-0.02941,0.19608) circle (0.00980);
\draw[very thin] (0.02941,-0.19608) circle (0.00980);
\draw[very thin] (-0.02941,-0.19608) circle (0.00980);
\draw[very thin] (0.52336,0.84112) circle (0.00935);
\draw[very thin] (0.52336,-0.84112) circle (0.00935);
\draw[very thin] (-0.52336,0.84112) circle (0.00935);
\draw[very thin] (-0.52336,-0.84112) circle (0.00935);
\draw[very thin] (0.07477,0.28037) circle (0.00935);
\draw[very thin] (-0.07477,0.28037) circle (0.00935);
\draw[very thin] (0.07477,-0.28037) circle (0.00935);
\draw[very thin] (-0.07477,-0.28037) circle (0.00935);
\draw[very thin] (0.82727,0.54545) circle (0.00909);
\draw[very thin] (0.62727,0.54545) circle (0.00909);
\draw[very thin] (0.82727,-0.54545) circle (0.00909);
\draw[very thin] (0.62727,-0.54545) circle (0.00909);
\draw[very thin] (-0.82727,0.54545) circle (0.00909);
\draw[very thin] (-0.62727,0.54545) circle (0.00909);
\draw[very thin] (-0.82727,-0.54545) circle (0.00909);
\draw[very thin] (-0.62727,-0.54545) circle (0.00909);
\draw[very thin] (0.04545,0.32727) circle (0.00909);
\draw[very thin] (-0.04545,0.32727) circle (0.00909);
\draw[very thin] (0.04545,-0.32727) circle (0.00909);
\draw[very thin] (-0.04545,-0.32727) circle (0.00909);
\draw[very thin] (0.43243,0.82883) circle (0.00901);
\draw[very thin] (0.43243,0.50450) circle (0.00901);
\draw[very thin] (0.43243,-0.82883) circle (0.00901);
\draw[very thin] (0.43243,-0.50450) circle (0.00901);
\draw[very thin] (-0.43243,0.82883) circle (0.00901);
\draw[very thin] (-0.43243,0.50450) circle (0.00901);
\draw[very thin] (-0.43243,-0.82883) circle (0.00901);
\draw[very thin] (-0.43243,-0.50450) circle (0.00901);
\draw[very thin] (0.13158,0.98246) circle (0.00877);
\draw[very thin] (0.13158,-0.98246) circle (0.00877);
\draw[very thin] (-0.13158,0.98246) circle (0.00877);
\draw[very thin] (-0.13158,-0.98246) circle (0.00877);
\draw[very thin] (0.13158,0.35088) circle (0.00877);
\draw[very thin] (-0.13158,0.35088) circle (0.00877);
\draw[very thin] (0.13158,-0.35088) circle (0.00877);
\draw[very thin] (-0.13158,-0.35088) circle (0.00877);
\draw[very thin] (0.33613,0.60504) circle (0.00840);
\draw[very thin] (0.33613,-0.60504) circle (0.00840);
\draw[very thin] (-0.33613,0.60504) circle (0.00840);
\draw[very thin] (-0.33613,-0.60504) circle (0.00840);
\draw[very thin] (0.31967,0.78689) circle (0.00820);
\draw[very thin] (0.31967,-0.78689) circle (0.00820);
\draw[very thin] (-0.31967,0.78689) circle (0.00820);
\draw[very thin] (-0.31967,-0.78689) circle (0.00820);
\draw[very thin] (0.97561,0.17886) circle (0.00813);
\draw[very thin] (0.97561,-0.17886) circle (0.00813);
\draw[very thin] (-0.97561,0.17886) circle (0.00813);
\draw[very thin] (-0.97561,-0.17886) circle (0.00813);
\draw[very thin] (0.92857,0.34921) circle (0.00794);
\draw[very thin] (0.92857,-0.34921) circle (0.00794);
\draw[very thin] (-0.92857,0.34921) circle (0.00794);
\draw[very thin] (-0.92857,-0.34921) circle (0.00794);
\draw[very thin] (0.85496,0.50382) circle (0.00763);
\draw[very thin] (0.24427,0.96183) circle (0.00763);
\draw[very thin] (0.36641,0.77863) circle (0.00763);
\draw[very thin] (0.85496,-0.50382) circle (0.00763);
\draw[very thin] (0.24427,-0.96183) circle (0.00763);
\draw[very thin] (0.36641,-0.77863) circle (0.00763);
\draw[very thin] (-0.85496,0.50382) circle (0.00763);
\draw[very thin] (-0.24427,0.96183) circle (0.00763);
\draw[very thin] (-0.36641,0.77863) circle (0.00763);
\draw[very thin] (-0.85496,-0.50382) circle (0.00763);
\draw[very thin] (-0.24427,-0.96183) circle (0.00763);
\draw[very thin] (-0.36641,-0.77863) circle (0.00763);
\draw[very thin] (0.73881,0.44776) circle (0.00746);
\draw[very thin] (0.26119,0.44776) circle (0.00746);
\draw[very thin] (0.73881,-0.44776) circle (0.00746);
\draw[very thin] (0.26119,-0.44776) circle (0.00746);
\draw[very thin] (-0.73881,0.44776) circle (0.00746);
\draw[very thin] (-0.26119,0.44776) circle (0.00746);
\draw[very thin] (-0.73881,-0.44776) circle (0.00746);
\draw[very thin] (-0.26119,-0.44776) circle (0.00746);
\draw[very thin] (0.88889,0.32593) circle (0.00741);
\draw[very thin] (0.88889,-0.32593) circle (0.00741);
\draw[very thin] (-0.88889,0.32593) circle (0.00741);
\draw[very thin] (-0.88889,-0.32593) circle (0.00741);
\draw[very thin] (0.76087,0.63768) circle (0.00725);
\draw[very thin] (0.76087,-0.63768) circle (0.00725);
\draw[very thin] (-0.76087,0.63768) circle (0.00725);
\draw[very thin] (-0.76087,-0.63768) circle (0.00725);
\draw[very thin] (0.55944,0.50350) circle (0.00699);
\draw[very thin] (0.22378,0.92308) circle (0.00699);
\draw[very thin] (0.55944,-0.50350) circle (0.00699);
\draw[very thin] (0.22378,-0.92308) circle (0.00699);
\draw[very thin] (-0.55944,0.50350) circle (0.00699);
\draw[very thin] (-0.22378,0.92308) circle (0.00699);
\draw[very thin] (-0.55944,-0.50350) circle (0.00699);
\draw[very thin] (-0.22378,-0.92308) circle (0.00699);
\draw[very thin] (0.00000,0.08392) circle (0.00699);
\draw[very thin] (0.00000,-0.08392) circle (0.00699);
\draw[very thin] (0.97945,0.16438) circle (0.00685);
\draw[very thin] (0.11644,0.98630) circle (0.00685);
\draw[very thin] (0.97945,-0.16438) circle (0.00685);
\draw[very thin] (0.11644,-0.98630) circle (0.00685);
\draw[very thin] (-0.97945,0.16438) circle (0.00685);
\draw[very thin] (-0.11644,0.98630) circle (0.00685);
\draw[very thin] (-0.97945,-0.16438) circle (0.00685);
\draw[very thin] (-0.11644,-0.98630) circle (0.00685);
\draw[very thin] (0.65306,0.74830) circle (0.00680);
\draw[very thin] (0.32653,0.47619) circle (0.00680);
\draw[very thin] (0.65306,-0.74830) circle (0.00680);
\draw[very thin] (0.32653,-0.47619) circle (0.00680);
\draw[very thin] (-0.65306,0.74830) circle (0.00680);
\draw[very thin] (-0.32653,0.47619) circle (0.00680);
\draw[very thin] (-0.65306,-0.74830) circle (0.00680);
\draw[very thin] (-0.32653,-0.47619) circle (0.00680);
\draw[very thin] (0.34000,0.93333) circle (0.00667);
\draw[very thin] (0.30000,0.50667) circle (0.00667);
\draw[very thin] (0.34000,-0.93333) circle (0.00667);
\draw[very thin] (0.30000,-0.50667) circle (0.00667);
\draw[very thin] (-0.34000,0.93333) circle (0.00667);
\draw[very thin] (-0.30000,0.50667) circle (0.00667);
\draw[very thin] (-0.34000,-0.93333) circle (0.00667);
\draw[very thin] (-0.30000,-0.50667) circle (0.00667);
\draw[very thin] (0.77419,0.42581) circle (0.00645);
\draw[very thin] (0.77419,-0.42581) circle (0.00645);
\draw[very thin] (-0.77419,0.42581) circle (0.00645);
\draw[very thin] (-0.77419,-0.42581) circle (0.00645);
\draw[very thin] (0.53797,0.83544) circle (0.00633);
\draw[very thin] (0.53797,-0.83544) circle (0.00633);
\draw[very thin] (-0.53797,0.83544) circle (0.00633);
\draw[very thin] (-0.53797,-0.83544) circle (0.00633);
\draw[very thin] (0.38889,0.49383) circle (0.00617);
\draw[very thin] (0.38889,-0.49383) circle (0.00617);
\draw[very thin] (-0.38889,0.49383) circle (0.00617);
\draw[very thin] (-0.38889,-0.49383) circle (0.00617);
\draw[very thin] (0.81437,0.50299) circle (0.00599);
\draw[very thin] (0.33533,0.71856) circle (0.00599);
\draw[very thin] (0.81437,-0.50299) circle (0.00599);
\draw[very thin] (0.33533,-0.71856) circle (0.00599);
\draw[very thin] (-0.81437,0.50299) circle (0.00599);
\draw[very thin] (-0.33533,0.71856) circle (0.00599);
\draw[very thin] (-0.81437,-0.50299) circle (0.00599);
\draw[very thin] (-0.33533,-0.71856) circle (0.00599);
\draw[very thin] (0.14371,0.35928) circle (0.00599);
\draw[very thin] (-0.14371,0.35928) circle (0.00599);
\draw[very thin] (0.14371,-0.35928) circle (0.00599);
\draw[very thin] (-0.14371,-0.35928) circle (0.00599);
\draw[very thin] (0.70000,0.70588) circle (0.00588);
\draw[very thin] (0.32353,0.56471) circle (0.00588);
\draw[very thin] (0.70000,-0.70588) circle (0.00588);
\draw[very thin] (0.32353,-0.56471) circle (0.00588);
\draw[very thin] (-0.70000,0.70588) circle (0.00588);
\draw[very thin] (-0.32353,0.56471) circle (0.00588);
\draw[very thin] (-0.70000,-0.70588) circle (0.00588);
\draw[very thin] (-0.32353,-0.56471) circle (0.00588);
\draw[very thin] (0.98246,0.15205) circle (0.00585);
\draw[very thin] (0.42105,0.90058) circle (0.00585);
\draw[very thin] (0.98246,-0.15205) circle (0.00585);
\draw[very thin] (0.42105,-0.90058) circle (0.00585);
\draw[very thin] (-0.98246,0.15205) circle (0.00585);
\draw[very thin] (-0.42105,0.90058) circle (0.00585);
\draw[very thin] (-0.98246,-0.15205) circle (0.00585);
\draw[very thin] (-0.42105,-0.90058) circle (0.00585);
\draw[very thin] (0.94828,0.29885) circle (0.00575);
\draw[very thin] (0.39655,0.52874) circle (0.00575);
\draw[very thin] (0.94828,-0.29885) circle (0.00575);
\draw[very thin] (0.39655,-0.52874) circle (0.00575);
\draw[very thin] (-0.94828,0.29885) circle (0.00575);
\draw[very thin] (-0.39655,0.52874) circle (0.00575);
\draw[very thin] (-0.94828,-0.29885) circle (0.00575);
\draw[very thin] (-0.39655,-0.52874) circle (0.00575);
\draw[very thin] (0.89385,0.43575) circle (0.00559);
\draw[very thin] (0.44693,0.50279) circle (0.00559);
\draw[very thin] (0.35754,0.56983) circle (0.00559);
\draw[very thin] (0.89385,-0.43575) circle (0.00559);
\draw[very thin] (0.44693,-0.50279) circle (0.00559);
\draw[very thin] (0.35754,-0.56983) circle (0.00559);
\draw[very thin] (-0.89385,0.43575) circle (0.00559);
\draw[very thin] (-0.44693,0.50279) circle (0.00559);
\draw[very thin] (-0.35754,0.56983) circle (0.00559);
\draw[very thin] (-0.89385,-0.43575) circle (0.00559);
\draw[very thin] (-0.44693,-0.50279) circle (0.00559);
\draw[very thin] (-0.35754,-0.56983) circle (0.00559);
\draw[very thin] (0.77473,0.46154) circle (0.00549);
\draw[very thin] (0.10440,0.98901) circle (0.00549);
\draw[very thin] (0.28022,0.85714) circle (0.00549);
\draw[very thin] (0.77473,-0.46154) circle (0.00549);
\draw[very thin] (0.10440,-0.98901) circle (0.00549);
\draw[very thin] (0.28022,-0.85714) circle (0.00549);
\draw[very thin] (-0.77473,0.46154) circle (0.00549);
\draw[very thin] (-0.10440,0.98901) circle (0.00549);
\draw[very thin] (-0.28022,0.85714) circle (0.00549);
\draw[very thin] (-0.77473,-0.46154) circle (0.00549);
\draw[very thin] (-0.10440,-0.98901) circle (0.00549);
\draw[very thin] (-0.28022,-0.85714) circle (0.00549);
\draw[very thin] (0.91803,0.28415) circle (0.00546);
\draw[very thin] (0.65574,0.48087) circle (0.00546);
\draw[very thin] (0.91803,-0.28415) circle (0.00546);
\draw[very thin] (0.65574,-0.48087) circle (0.00546);
\draw[very thin] (-0.91803,0.28415) circle (0.00546);
\draw[very thin] (-0.65574,0.48087) circle (0.00546);
\draw[very thin] (-0.91803,-0.28415) circle (0.00546);
\draw[very thin] (-0.65574,-0.48087) circle (0.00546);
\draw[very thin] (0.82258,0.55914) circle (0.00538);
\draw[very thin] (0.30645,0.94624) circle (0.00538);
\draw[very thin] (0.82258,-0.55914) circle (0.00538);
\draw[very thin] (0.30645,-0.94624) circle (0.00538);
\draw[very thin] (-0.82258,0.55914) circle (0.00538);
\draw[very thin] (-0.30645,0.94624) circle (0.00538);
\draw[very thin] (-0.82258,-0.55914) circle (0.00538);
\draw[very thin] (-0.30645,-0.94624) circle (0.00538);
\draw[very thin] (0.67016,0.69110) circle (0.00524);
\draw[very thin] (0.25131,0.43979) circle (0.00524);
\draw[very thin] (0.67016,-0.69110) circle (0.00524);
\draw[very thin] (0.25131,-0.43979) circle (0.00524);
\draw[very thin] (-0.67016,0.69110) circle (0.00524);
\draw[very thin] (-0.25131,0.43979) circle (0.00524);
\draw[very thin] (-0.67016,-0.69110) circle (0.00524);
\draw[very thin] (-0.25131,-0.43979) circle (0.00524);
\draw[very thin] (0.66495,0.61856) circle (0.00515);
\draw[very thin] (0.48969,0.86598) circle (0.00515);
\draw[very thin] (0.33505,0.61856) circle (0.00515);
\draw[very thin] (0.66495,-0.61856) circle (0.00515);
\draw[very thin] (0.48969,-0.86598) circle (0.00515);
\draw[very thin] (0.33505,-0.61856) circle (0.00515);
\draw[very thin] (-0.66495,0.61856) circle (0.00515);
\draw[very thin] (-0.48969,0.86598) circle (0.00515);
\draw[very thin] (-0.33505,0.61856) circle (0.00515);
\draw[very thin] (-0.66495,-0.61856) circle (0.00515);
\draw[very thin] (-0.48969,-0.86598) circle (0.00515);
\draw[very thin] (-0.33505,-0.61856) circle (0.00515);
\draw[very thin] (0.73846,0.66667) circle (0.00513);
\draw[very thin] (0.73846,-0.66667) circle (0.00513);
\draw[very thin] (-0.73846,0.66667) circle (0.00513);
\draw[very thin] (-0.73846,-0.66667) circle (0.00513);
\draw[very thin] (0.00000,0.07179) circle (0.00513);
\draw[very thin] (0.00000,-0.07179) circle (0.00513);
\draw[very thin] (0.98485,0.14141) circle (0.00505);
\draw[very thin] (0.98485,-0.14141) circle (0.00505);
\draw[very thin] (-0.98485,0.14141) circle (0.00505);
\draw[very thin] (-0.98485,-0.14141) circle (0.00505);
\draw[very thin] (0.01515,0.14141) circle (0.00505);
\draw[very thin] (-0.01515,0.14141) circle (0.00505);
\draw[very thin] (0.01515,-0.14141) circle (0.00505);
\draw[very thin] (-0.01515,-0.14141) circle (0.00505);
\draw[very thin] (0.82759,0.38424) circle (0.00493);
\draw[very thin] (0.19704,0.97537) circle (0.00493);
\draw[very thin] (0.82759,-0.38424) circle (0.00493);
\draw[very thin] (0.19704,-0.97537) circle (0.00493);
\draw[very thin] (-0.82759,0.38424) circle (0.00493);
\draw[very thin] (-0.19704,0.97537) circle (0.00493);
\draw[very thin] (-0.82759,-0.38424) circle (0.00493);
\draw[very thin] (-0.19704,-0.97537) circle (0.00493);
\draw[very thin] (0.03941,0.20690) circle (0.00493);
\draw[very thin] (-0.03941,0.20690) circle (0.00493);
\draw[very thin] (0.03941,-0.20690) circle (0.00493);
\draw[very thin] (-0.03941,-0.20690) circle (0.00493);
\draw[very thin] (0.90777,0.40777) circle (0.00485);
\draw[very thin] (0.64563,0.75728) circle (0.00485);
\draw[very thin] (0.90777,-0.40777) circle (0.00485);
\draw[very thin] (0.64563,-0.75728) circle (0.00485);
\draw[very thin] (-0.90777,0.40777) circle (0.00485);
\draw[very thin] (-0.64563,0.75728) circle (0.00485);
\draw[very thin] (-0.90777,-0.40777) circle (0.00485);
\draw[very thin] (-0.64563,-0.75728) circle (0.00485);
\draw[very thin] (0.07143,0.26667) circle (0.00476);
\draw[very thin] (-0.07143,0.26667) circle (0.00476);
\draw[very thin] (0.07143,-0.26667) circle (0.00476);
\draw[very thin] (-0.07143,-0.26667) circle (0.00476);
\draw[very thin] (0.63256,0.55814) circle (0.00465);
\draw[very thin] (0.18605,0.94884) circle (0.00465);
\draw[very thin] (0.48372,0.83721) circle (0.00465);
\draw[very thin] (0.63256,-0.55814) circle (0.00465);
\draw[very thin] (0.18605,-0.94884) circle (0.00465);
\draw[very thin] (0.48372,-0.83721) circle (0.00465);
\draw[very thin] (-0.63256,0.55814) circle (0.00465);
\draw[very thin] (-0.18605,0.94884) circle (0.00465);
\draw[very thin] (-0.48372,0.83721) circle (0.00465);
\draw[very thin] (-0.63256,-0.55814) circle (0.00465);
\draw[very thin] (-0.18605,-0.94884) circle (0.00465);
\draw[very thin] (-0.48372,-0.83721) circle (0.00465);
\draw[very thin] (0.26147,0.88073) circle (0.00459);
\draw[very thin] (0.26147,-0.88073) circle (0.00459);
\draw[very thin] (-0.26147,0.88073) circle (0.00459);
\draw[very thin] (-0.26147,-0.88073) circle (0.00459);
\draw[very thin] (0.03211,0.33028) circle (0.00459);
\draw[very thin] (-0.03211,0.33028) circle (0.00459);
\draw[very thin] (0.03211,-0.33028) circle (0.00459);
\draw[very thin] (-0.03211,-0.33028) circle (0.00459);
\draw[very thin] (0.54795,0.50228) circle (0.00457);
\draw[very thin] (0.54795,0.83105) circle (0.00457);
\draw[very thin] (0.54795,-0.50228) circle (0.00457);
\draw[very thin] (0.54795,-0.83105) circle (0.00457);
\draw[very thin] (-0.54795,0.50228) circle (0.00457);
\draw[very thin] (-0.54795,0.83105) circle (0.00457);
\draw[very thin] (-0.54795,-0.50228) circle (0.00457);
\draw[very thin] (-0.54795,-0.83105) circle (0.00457);
\draw[very thin] (0.10959,0.31963) circle (0.00457);
\draw[very thin] (-0.10959,0.31963) circle (0.00457);
\draw[very thin] (0.10959,-0.31963) circle (0.00457);
\draw[very thin] (-0.10959,-0.31963) circle (0.00457);
\draw[very thin] (0.77027,0.63063) circle (0.00450);
\draw[very thin] (0.09459,0.99099) circle (0.00450);
\draw[very thin] (0.77027,-0.63063) circle (0.00450);
\draw[very thin] (0.09459,-0.99099) circle (0.00450);
\draw[very thin] (-0.77027,0.63063) circle (0.00450);
\draw[very thin] (-0.09459,0.99099) circle (0.00450);
\draw[very thin] (-0.77027,-0.63063) circle (0.00450);
\draw[very thin] (-0.09459,-0.99099) circle (0.00450);
\draw[very thin] (0.09459,0.34234) circle (0.00450);
\draw[very thin] (-0.09459,0.34234) circle (0.00450);
\draw[very thin] (0.09459,-0.34234) circle (0.00450);
\draw[very thin] (-0.09459,-0.34234) circle (0.00450);
\draw[very thin] (0.98678,0.13216) circle (0.00441);
\draw[very thin] (0.42291,0.81938) circle (0.00441);
\draw[very thin] (0.98678,-0.13216) circle (0.00441);
\draw[very thin] (0.42291,-0.81938) circle (0.00441);
\draw[very thin] (-0.98678,0.13216) circle (0.00441);
\draw[very thin] (-0.42291,0.81938) circle (0.00441);
\draw[very thin] (-0.98678,-0.13216) circle (0.00441);
\draw[very thin] (-0.42291,-0.81938) circle (0.00441);
\draw[very thin] (0.96087,0.26087) circle (0.00435);
\draw[very thin] (0.84783,0.36522) circle (0.00435);
\draw[very thin] (0.96087,-0.26087) circle (0.00435);
\draw[very thin] (0.84783,-0.36522) circle (0.00435);
\draw[very thin] (-0.96087,0.26087) circle (0.00435);
\draw[very thin] (-0.84783,0.36522) circle (0.00435);
\draw[very thin] (-0.96087,-0.26087) circle (0.00435);
\draw[very thin] (-0.84783,-0.36522) circle (0.00435);
\draw[very thin] (0.15217,0.36522) circle (0.00435);
\draw[very thin] (-0.15217,0.36522) circle (0.00435);
\draw[very thin] (0.15217,-0.36522) circle (0.00435);
\draw[very thin] (-0.15217,-0.36522) circle (0.00435);
\draw[very thin] (0.72727,0.45022) circle (0.00433);
\draw[very thin] (0.72727,0.64069) circle (0.00433);
\draw[very thin] (0.31169,0.91775) circle (0.00433);
\draw[very thin] (0.31169,0.79654) circle (0.00433);
\draw[very thin] (0.72727,-0.45022) circle (0.00433);
\draw[very thin] (0.72727,-0.64069) circle (0.00433);
\draw[very thin] (0.31169,-0.91775) circle (0.00433);
\draw[very thin] (0.31169,-0.79654) circle (0.00433);
\draw[very thin] (-0.72727,0.45022) circle (0.00433);
\draw[very thin] (-0.72727,0.64069) circle (0.00433);
\draw[very thin] (-0.31169,0.91775) circle (0.00433);
\draw[very thin] (-0.31169,0.79654) circle (0.00433);
\draw[very thin] (-0.72727,-0.45022) circle (0.00433);
\draw[very thin] (-0.72727,-0.64069) circle (0.00433);
\draw[very thin] (-0.31169,-0.91775) circle (0.00433);
\draw[very thin] (-0.31169,-0.79654) circle (0.00433);
\draw[very thin] (0.44872,0.88889) circle (0.00427);
\draw[very thin] (0.44872,-0.88889) circle (0.00427);
\draw[very thin] (-0.44872,0.88889) circle (0.00427);
\draw[very thin] (-0.44872,-0.88889) circle (0.00427);
\draw[very thin] (0.93724,0.25105) circle (0.00418);
\draw[very thin] (0.93724,-0.25105) circle (0.00418);
\draw[very thin] (-0.93724,0.25105) circle (0.00418);
\draw[very thin] (-0.93724,-0.25105) circle (0.00418);
\draw[very thin] (0.86364,0.49587) circle (0.00413);
\draw[very thin] (0.59091,0.49587) circle (0.00413);
\draw[very thin] (0.86364,-0.49587) circle (0.00413);
\draw[very thin] (0.59091,-0.49587) circle (0.00413);
\draw[very thin] (-0.86364,0.49587) circle (0.00413);
\draw[very thin] (-0.59091,0.49587) circle (0.00413);
\draw[very thin] (-0.86364,-0.49587) circle (0.00413);
\draw[very thin] (-0.59091,-0.49587) circle (0.00413);
\fill (0.50000,0.00000) circle (0.012);
\fill (-0.50000,0.00000) circle (0.012);
\fill (0.00000,0.66667) circle (0.012);
\fill (-0.00000,-0.66667) circle (0.012);
\fill (0.50000,0.66667) circle (0.012);
\fill (0.72727,0.54545) circle (0.007);
\fill (0.83333,0.44444) circle (0.007);
\fill (0.88889,0.37037) circle (0.007);
\fill (0.92105,0.31579) circle (0.007);
\fill (0.70000,0.66667) circle (0.007);
\fill (0.61538,0.51282) circle (0.007);
\fill (0.35714,0.85714) circle (0.007);
\fill (0.26923,0.92308) circle (0.007);
\fill (0.45714,0.85714) circle (0.007);
\fill (0.34783,0.52174) circle (0.007);
\fill (0.50000,-0.66667) circle (0.012);
\fill (0.72727,-0.54545) circle (0.007);
\fill (0.83333,-0.44444) circle (0.007);
\fill (0.88889,-0.37037) circle (0.007);
\fill (0.92105,-0.31579) circle (0.007);
\fill (0.70000,-0.66667) circle (0.007);
\fill (0.61538,-0.51282) circle (0.007);
\fill (0.35714,-0.85714) circle (0.007);
\fill (0.26923,-0.92308) circle (0.007);
\fill (0.45714,-0.85714) circle (0.007);
\fill (0.34783,-0.52174) circle (0.007);
\fill (-0.50000,0.66667) circle (0.012);
\fill (-0.72727,0.54545) circle (0.007);
\fill (-0.83333,0.44444) circle (0.007);
\fill (-0.88889,0.37037) circle (0.007);
\fill (-0.92105,0.31579) circle (0.007);
\fill (-0.70000,0.66667) circle (0.007);
\fill (-0.61538,0.51282) circle (0.007);
\fill (-0.35714,0.85714) circle (0.007);
\fill (-0.26923,0.92308) circle (0.007);
\fill (-0.45714,0.85714) circle (0.007);
\fill (-0.34783,0.52174) circle (0.007);
\fill (-0.50000,-0.66667) circle (0.012);
\fill (-0.72727,-0.54545) circle (0.007);
\fill (-0.83333,-0.44444) circle (0.007);
\fill (-0.88889,-0.37037) circle (0.007);
\fill (-0.92105,-0.31579) circle (0.007);
\fill (-0.70000,-0.66667) circle (0.007);
\fill (-0.61538,-0.51282) circle (0.007);
\fill (-0.35714,-0.85714) circle (0.007);
\fill (-0.26923,-0.92308) circle (0.007);
\fill (-0.45714,-0.85714) circle (0.007);
\fill (-0.34783,-0.52174) circle (0.007);
\fill (0.00000,0.26667) circle (0.007);
\fill (0.00000,0.17143) circle (0.007);
\fill (0.07895,0.31579) circle (0.007);
\fill (-0.07895,0.31579) circle (0.007);
\fill (0.00000,-0.26667) circle (0.007);
\fill (0.00000,-0.17143) circle (0.007);
\fill (0.07895,-0.31579) circle (0.007);
\fill (-0.07895,-0.31579) circle (0.007);
\draw[thick] (0,0) circle (1);
\fill (0,0) circle (0.022) node[left=2pt] {$0$};
%\node[right] at (-1.33,1.2) {$F\equiv0$ on $|z|>1$};
%\draw[->] (0.6,0.1) -- (0.45,-1.4);
%\node[font=\small] at (0,-1.57) {$F=\frac1z-\frac{1}{4(z-\frac12)}-\frac12$ on $D(\frac12,\frac12)\setminus\{\frac12\}$};
% axes
\draw[->,thin] (-1.3,0) -- (1.3,0);
\draw[->,thin] (0,-1.3) -- (0,1.3);
\fill (1,0) circle (0.022) node[below right=-1pt] {$1$};
\fill (0,1) circle (0.022) node[above left=-1pt] {$i$};
\end{tikzpicture}

%% file: Monogenic_submitted_8_10_2026.bbl
\providecommand{\bysame}{\leavevmode\hbox to3em{\hrulefill}\thinspace}
\providecommand{\MR}{\relax\ifhmode\unskip\space\fi MR }
% \MRhref is called by the amsart/book/proc definition of \MR.
\providecommand{\MRhref}[2]{%
  \href{http://www.ams.org/mathscinet-getitem?mr=#1}{#2}
}
\providecommand{\href}[2]{#2}
\begin{thebibliography}{10}

\bibitem{Arnold_SD1}
V.I. Arnold, \emph{{Small denominators I, Mapping the circle onto itself}},
  Izv. Akad. Nauk. Math. \textbf{25} (1961), no.~1, 25--86, English
  translation: Translations Amer. Math. Soc., 2nd series, 46, p. 213-284.

\bibitem{Perez-Marco_Biswas}
K.~Biswas and R.~Perez-Marco, \emph{{Log-Riemann surfaces}}, arXiv:1512.03776,
  2015.

\bibitem{Borel_1898}
{\'E}.~Borel, \emph{Le{\c{c}}ons sur la th{\'e}orie des fonctions},
  Gauthier-Villars, Paris, 1898.

\bibitem{Borel_monogene}
\bysame, \emph{Les fonctions monog\`enes non analytiques}, {Bulletin de la
  Soci\'et\'e Math\'ematique de France} \textbf{40} (1912), 205--219,
  http://www.numdam.org/articles/10.24033/bsmf.902/.

\bibitem{Borel1917}
\bysame, \emph{Le{\c{c}}ons sur les fonctions monog{\`e}nes uniformes d'une
  variable complexe}, vol.~21, Gauthier-Villars et Cie, 1917.

\bibitem{Carleman_1926}
T.~Carleman, \emph{{Les Fonctions quasi analytiques: le{\c{c}}ons
  profess{\'e}es au Coll\`ege de France}}, Gauthier-Villars, 1926.

\bibitem{Denjoy_1921}
A.~Denjoy, \emph{{Sur les fonctions quasi-analytiques de variable r\'eelle }},
  Comptes Rendus \`a l'Acad\'emie des Sciences (1921), 1329--1331.

\bibitem{Ecalle_fonctions}
J.~\'Ecalle, \emph{Les fonctions r\'esurgentes, vol. 1, alg\`ebres de fonctions
  r\'esurgentes}, Publications Math\'ematiques d'Orsay, no. 81-05, Universit\'e
  de Paris-Sud, Orsay, 1981, 248 pages.

\bibitem{Euler_log}
L.~Euler, \emph{Consideratio quarumdam serierum, quae singularibus
  proprietatibus sunt praeditae}, {Novi Commentarii academiae scientiarum
  Petropolitanae} \textbf{3} (1753), 86--108, E190, English translation
  available on
  https://www.agtz.mathematik.uni-mainz.de/algebraische-geometrie/van-straten/euler-kreis-mainz/.

\bibitem{Versal_fields}
M.~Garay and D.~van Straten, \emph{{Versal deformations of isolated vector
  field singularities}}, arXiv:2011.06802, 2020.

\bibitem{HNF}
\bysame, \emph{The {Hamiltonian} normal form}, Journal of Geometry and Physics
  \textbf{227} (2026), 105881.

\bibitem{Garoufalidis_Zagier}
S.~Garoufalidis and D.~Zagier, \emph{{Knots and Their Related $q$--series}},
  {Sigma} \textbf{19} (2023), Paper No. 082, 39 pages.

\bibitem{Gasper_Rahman}
G.~Gasper and M.~Rahman, \emph{Basic hypergeometric series}, vol.~96, Cambridge
  university press, 2011.

\bibitem{Habiro2004}
K.~Habiro, \emph{Cyclotomic completions of polynomial rings}, Publications of
  the Research Institute for Mathematical Sciences \textbf{40} (2004), no.~4,
  1127--1146.

\bibitem{Herman1985}
M.R. Herman, \emph{Simple proofs of local conjugacy theorems for
  diffeomorphisms of the circle with almost every rotation number}, {Boletim da
  Sociedade Brasileira de Matem{\'a}tica} \textbf{16} (1985), no.~1, 45--83.

\bibitem{Jarnik_1928}
V.~Jarn{\'\i}k, \emph{Zur metrischen {T}heorie der diophantischen
  {A}pproximationen}, Prace Matematyczno-Fizyczne \textbf{36} (1928--1929),
  91--106.

\bibitem{Malgrange_resommabilite}
B.~Malgrange, \emph{Sommation des s\'eries divergentes}, Expositiones
  Mathematicae \textbf{13} (1995), no.~2/3, 163--222.

\bibitem{Marmi_Sauzin}
S.~Marmi and D.~Sauzin, \emph{Quasianalytic monogenic solutions of a
  cohomological equation}, vol. 164, {Memoirs of the American Mathematical
  Society}, no. 780, American Mathematical Soc., 2003.

\bibitem{Mourtada_Ramanujan}
H.~Mourtada, \emph{{Hilbert Meets Ramanujan: Singularity Theory and Integer
  Partitions}}, Bull. Amer. Math. Soc. (N.S.) \textbf{62} (2025), no.~1,
  93--111.

\bibitem{Perez_Marco_2020_Cantor}
R.~P{\'e}rez-Marco, \emph{{The Cantor Riemannium}}, arXiv:2004.10541, 2020.

\bibitem{Poincare_Methodes_2}
H.~Poincar\'e, \emph{Les m\'ethodes nouvelles de la m\'ecanique c\'eleste},
  vol.~2, Gauthier-Villars, 1893.

\bibitem{Popov_KAM}
G.~Popov, \emph{{Invariant tori, Effective Stability, and Quasimodes with
  Exponentially Small error Terms I- Birkhoff Normal Forms}}, Ann. Henri
  Poincar\'e \textbf{1} (2004), no.~2, 223--248.

\bibitem{Runge}
C.~Runge, \emph{{Zur Theorie der eindeutigen analytischen Funktionen}}, Acta
  Mathematica \textbf{6} (1885), 229--244.

\bibitem{Sondow_Zudilin}
J.~Sondow and W.~Zudilin, \emph{{Euler's constant, q-logarithms, and formulas
  of Ramanujan and Gosper}}, The Ramanujan Journal \textbf{12} (2006), no.~2,
  225--244.

\bibitem{Stieltjes_1894}
T.-J. Stieltjes, \emph{{Recherches sur les fractions continues}}, Annales de la
  Facult\'e des sciences de Toulouse pour les sciences math\'ematiques et les
  sciences physiques \textbf{8} (1894), no.~4, 1--122,
  http://www.numdam.org/item/AFST\_1894\_1\_8\_4\_J1\_0.

\bibitem{Weierstrass_1880}
K.~Weierstrass, \emph{Zur {F}unctionenlehre}, Monatsberichte der K{\"o}niglich
  Preussischen Akademie der Wissenschaften zu Berlin (1880), 719--743,
  Reprinted in Mathematische Werke, vol. 2, Mayer \& M{\"u}ller, Berlin, 1895,
  201--223.

\bibitem{Weierstrass_1881}
\bysame, \emph{Remarques sur quelques points de la th{\'e}orie des fonctions
  analytiques}, Bulletin des Sciences Math{\'e}matiques et Astronomiques (2)
  \textbf{5} (1881), 157--183, French translation of the preceding by J.
  Tannery.

\bibitem{Whitney_extension}
H.~Whitney, \emph{Analytic extensions of differentiable functions defined in
  closed sets}, Trans. Amer. math. Soc. \textbf{36} (1934), 63--89.

\bibitem{Winkler}
J.~Winkler, \emph{A uniqueness theorem for monogenic functions}, Ann. Acad.
  Sci. Fenn. Ser. AI Math \textbf{18} (1993), no.~1, 105--116.

\bibitem{Wolff_1921}
J.~Wolff, \emph{Sur les s\'eries $\sum a_k/(z-\alpha_k)$}, C. R. Acad. Sci.
  Paris \textbf{173} (1921), 1056--1057, 1327--1328.

\bibitem{Zagier2001}
D.~Zagier, \emph{{Vassiliev invariants and a strange identity related to the
  Dedekind eta-function}}, Topology \textbf{40} (2001), no.~5, 945--960.

\bibitem{Zalcman1968}
L.~Zalcman, \emph{Analytic capacity and rational approximation}, Lecture Notes
  in Mathematics, vol.~50, Springer, 1968.

\end{thebibliography}
